\documentclass[a4paper, 10pt]{article}

\usepackage{amsmath}
\usepackage{amsthm}
\usepackage{thmtools}
\usepackage{amssymb}
\usepackage{mathtools}
\usepackage{verbatim}
\usepackage{enumerate}
\usepackage[hidelinks]{hyperref}
\usepackage[inner=2.4cm,outer=2.4cm,bottom=3.5cm,top=3.5cm]{geometry}
\usepackage{subcaption}
\usepackage[nameinlink,capitalise]{cleveref}
\usepackage{booktabs}
\usepackage[subrefformat=parens,labelformat=parens]{subcaption}

\usepackage[backend=biber,style=trad-abbrv,
doi=false, 
isbn=false,
url=false,
eprint=false]{biblatex}
\newcommand{\email}[1]{\texttt{\href{mailto:#1}{#1}}}

\newtheorem{theorem}{Theorem}[section]

\newtheorem{lemma}[theorem]{Lemma}

\newtheorem{corollary}[theorem]{Corollary}

\theoremstyle{remark}
\newtheorem{remark}[theorem]{Remark}
\theoremstyle{definition}

\numberwithin{equation}{section}

\newcommand{\R}{\ensuremath{\mathbb{R}}}
\newcommand{\N}{\ensuremath{\mathbb{N}}}
\newcommand{\Z}{\ensuremath{\mathbb{Z}}}

\newcommand{\dd}{\mathrm{d}}

\newcommand{\dx}{\dd x}
\newcommand{\dy}{\dd y}

\newcommand{\defeq}{\coloneq}
\newcommand{\triangulation}{\mathcal T_h}
\newcommand{\plinear}[1]{\mathrm{PL}(#1)}
\newcommand{\dist}{\mathrm{dist}}
\newcommand{\dOmega}{d_\Omega}
\newcommand{\Rd}{{\mathbb{R}^d}}

\newcommand{\triplenorm}[1]{%
  \left|\!\left|\!\left| #1 \right|\!\right|\!\right|%
}

\newcommand{\subjectclassification}[1]{

	{\small\textbf{\textit{AMS Subject Classification --- }} #1}

}
\newcommand{\keywords}[1]{

	{\small\textbf{\textit{Keywords --- }} #1}

}

\usepackage[dvipsnames]{xcolor}

\begin{document}

\title{Finite Elements with weighted bases for the fractional Laplacian
}

\author{F\'elix del Teso%
    \thanks{Departamento de Matemáticas, Universidad Autónoma de Madrid. 28049 Madrid (Spain). \email{felix.delteso@uam.es}} %
    \and Stefano Fronzoni%
    \thanks{Mathematical Institute, University of Oxford, OX2 6GG Oxford, United Kingdom, \email{fronzoni@maths.ox.ac.uk}} %
    \and%
    David Gómez-Castro%
    \thanks{Departamento de Matemáticas, Universidad Autónoma de Madrid. 28049 Madrid (Spain). \email{david.gomezcastro@uam.es}}
    \thanks{Instituto de Ciencias Matemáticas, Consejo Superior de Investigaciones Científicas. 28049 Madrid (Spain)}
}

\maketitle

\begin{abstract}\noindent
    This work presents a numerical study of the Dirichlet problem for the fractional Laplacian $(-\Delta)^s$ with $s\in(0,1)$ using Finite Element methods with non-standard bases. Classical approaches based on piece-wise linear basis yield $h^{\frac 1 2}$ convergence rates in the Sobolev–Slobodeckij norm $H^s$ due to the limited boundary regularity of the solution $u(x)$, which behaves like $\textup{dist}(x,\mathbb{R}^d\setminus \Omega)^s$, where $h$ is the diameter of the mesh elements. To overcome this limitation, we propose a novel Finite Element basis of the form $\delta^s~\times~($piece-wise linear functions$)$, where $\delta$ is any suitably smooth approximation of $\textup{dist}(x,\mathbb{R}^d\setminus \Omega)$. This exploits the improved regularity of $u/\delta^s$, achieving higher convergence rates. Under standard smoothness assumptions the method attains an order $h^{2-s}$ on quasi-uniform meshes, improving the rates with the piece-wise linear basis. We provide a rigorous theoretical error analysis with explicit rates and validate it through numerical experiments.
\end{abstract}

\subjectclassification{35K55; 35R11; 65N30.}

\tableofcontents

\newpage

\section{Introduction and main results}

The objective of this paper is to perform a numerical study, using finite element methods with non-standard bases, of the following Dirichlet problem associated with the fractional Laplacian:
\begin{align}\label{eq:fracDirProb}
    \left\{
    \begin{aligned}
        (-\Delta)^s u(x) & = f(x), &  & x \in \Omega,                        \\
        u(x)             & = 0,    &  & x \in \mathbb{R}^d \setminus \Omega,
    \end{aligned}\right.
\end{align}
with $\Omega$ a bounded open subset of $\R^d$, $f\in L^\infty(\R^d)$, and $(-\Delta)^s$ denoting the fractional Laplacian of order $s\in(0,1)$ and given by
\begin{equation} \label{eq:FracLapInt}
    (-\Delta)^s u(x) = C_{d,s} \textrm{P.V.} \int_{\mathbb{R}^d} \frac{u(x)-u(y)}{|x-y|^{d+2s}} \dy, \quad \textup{with } C_{d,s} = \frac{4^s \Gamma(d/2 + s)}{\pi^{d/2} |\Gamma(-s)|}.
\end{equation}

Finite element methods for problem \eqref{eq:fracDirProb} have been studied using piece-wise linear basis functions on polygonal domains (see, for instance, Acosta \& Borthagaray \cite{Acosta2017784, Borthagaray2017}). The reported convergence rates for these methods, in $s$-fractional Sobolev-Slobodeckij norm, are at best approximately $h^{\frac{1}{2}}$, where $h$ is the diameter of the mesh triangles.
The sublinear convergence rates are due to the regularity of the solution $u$, which is at most $u\in H^{s+\frac{1}{2}-\varepsilon}(\Omega)$ for small $\varepsilon>0$ (we refer the reader to Proposition 3.6 and Proposition 3.11 in \cite{Borthagaray2017}). This reduced regularity arises from the fact that, even in smooth domains, the solution $u$ behaves like $d_\Omega^s$ in a neighbourhood of the boundary $\partial \Omega$, where $\dOmega(x) = \dist(x, \R^d\setminus \Omega)$. We refer to \cite{AbatangeloRos-Oton2020} for a recent result of this type and Section 2.7 in \cite{fernandez-realIntegroDifferentialEllipticEquations2024} for a summary of the related regularity results. The low regularity of $u$ makes higher convergence rates difficult to obtain using a standard piece-wise linear basis that is normally used for solutions that are at least weakly differentiable once.
In \cite{Bonito2019}, the authors point out that when \(u\) is sufficiently regular and linear up to the boundary (e.g., \(u(x) = (1 - |x|^2)_+\) with \(\Omega = B_1\)), the higher-order convergence rate \(h^{\frac{3}{2} - s}\) in the \(s\)-fractional Sobolev–Slobodeckij norm can be achieved using piecewise linear basis functions. It is worth noting that such functions \(u\) correspond to solutions of problems with a sign-changing right-hand side \(f\) (see further discussion in \Cref{sec:extOp}).

\medskip

Our approach builds on the observation that, intuitively, if $\delta$ denotes a smooth approximation of the distance function $d_\Omega$, then the quotient $u/\delta^s$ exhibits significantly better regularity up to the boundary, provided that both $f$ and the domain $\Omega$ are sufficiently smooth. This regularity theory has been developed through a series of contributions
\cite{RosOtonSerra2014, Grubb2015, RosOton2016, RosOton2016_2, RosOtonSerra2017, AbatangeloRos-Oton2020, Grubb2023} (see \cite{fernandez-realIntegroDifferentialEllipticEquations2024} for a good survey on the matter).
Motivated by this fact, we construct a tailored finite element basis that captures the improved regularity of $u/\delta^s$, thereby achieving higher-order convergence rates than those available in the existing literature.
The idea, that will be detailed in \eqref{eq:Vh}, is to use a basis of the form $\delta^s \times ($piece-wise linear functions$)$.
One of the key advantages of the finite element method is its effectiveness in approximating smooth functions and our method takes advantage of this well-known fact in order to approximate $u/\delta^s$ and subsequently, knowing $\delta^s$, the solution $u$ itself. This motivates our strategy of approximating $u/\delta^s$ rather than $u$ directly.
The advantage of this approach is that we can obtain improved convergence rates, without requiring higher regularity of $u$, but rather $f$ and $\Omega$ sufficiently smooth, which is often the case in applications.  Assuming enough regularity on the right-hand side $f$ and the domain $\Omega$, we obtain convergence rates of order $h^{2-s}$ in $s$-fractional Sobolev-Slobodeckij norm, improving all the previously reported results in the literature.
Using Aubin-Nitsche duality, we obtain $L^2$-rates of higher than our $H^s$-rates.
At the end of the paper we provide numerical experiments that precisely match the $H^s$-rate of order $h^{2-s}$, and that suggest the $L^2$-rate is of order $h^2$.

\subsection{Preliminaries and functional setting}
In this section we introduce the functional spaces, the weak formulation of the fractional Dirichlet problem, as well as the finite element formulation of our numerical scheme with a weighted basis.

\paragraph{Sobolev and  Hölder  spaces.} We denote by $L^p(\Omega)$, for $p \in [1,\infty]$, the standard Lebesgue spaces endowed with the norm $\|\cdot\|_{L^{p}(\Omega)}$.
We recall the fractional Sobolev-Slobodeckij spaces of order $s \in (0,1)$, where we will denote the bilinear form
\begin{equation} \label{eq:BinFormFracLap}
    a(u,v) \defeq \frac{C_{d,s}}{2}
    \iint_{\mathbb{R}^d \times \mathbb{R}^d}
    \frac{(u(x) - u(y))(v(x) - v(y))}{|x-y|^{d+2s}}
    \, \dx \, \dy,
\end{equation}
and the associated seminorm $
    [w]_{H^{s}(\mathbb{R}^d)} \defeq a(w,w)^{1/2}$. Accordingly, for $s \in (0,1)$ we define
\[
    H^s(\mathbb{R}^d) \defeq
    \left\{ w \in L^2(\mathbb{R}^d) : [w]_{H^{s}(\mathbb{R}^d)} < \infty \right\},
\]
endowed with the norm
$\|w\|_{H^{s}(\mathbb{R}^d)} \defeq
    \|w\|_{L^2(\mathbb{R}^d)} + [w]_{H^{s}(\mathbb{R}^d)}$.
We also introduce the corresponding energy space
\begin{equation*}
    V \defeq \{ w \in H^s(\Rd) \text{ such that } w = 0 \text{ in } \Rd \setminus \Omega \},
\end{equation*}
which is sometimes denoted by \(\widetilde H^s(\Omega)\) in the literature.
\begin{remark}
    It is standard to define the trace spaces for \( s \in (0,1) \) as $H^s_0(\Omega) \defeq
        \overline{C_c^\infty(\Omega)}^{H^s}$.
    When \(\partial \Omega \in C^{0,1}\) and \(s \ne \tfrac{1}{2}\), we have the alternative characterization $V = H^s_0(\Omega)$ (see, e.g., \cite{hitchhiker2012}).
    In the case \(s = \tfrac{1}{2}\), the space \(V\) coincides with the so-called Lions–Magenes space \(H^{1/2}_{00}(\Omega)\).
\end{remark}
The well-known fractional Poincaré inequality
\begin{equation} \label{eq:fractional Poincare}
    \| u \|_{L^2(\Omega)} \le C [u]_{H^s(\Rd)}, \qquad \forall u \in V,
\end{equation}
ensures that the seminorm \([\cdot]_{H^s(\R^d)}\) is in fact an equivalent norm on \(V\). For further details, we refer the reader to \cite{hitchhiker2012, David2021}.
Finally, for $\alpha>0$ we denote by $W^{\alpha, \infty}(\Omega)$ the Hölder space of order $\alpha$
\begin{equation*}
    W^{\alpha, \infty} (\Omega) \defeq C^{\lceil \alpha \rceil - 1, \alpha - \lceil \alpha \rceil + 1} (\overline \Omega).
\end{equation*}
We recall that Morrey's inequality guarantees for all $\alpha> 0$ we have  $H^{\alpha+\frac{d}{2}}(\R^d) \subset W^{\alpha,\infty}_{\text{loc}}(\Rd)$.

\paragraph{Weak formulation of the fractional Dirichlet problem.}
We point out that we have the integration by parts formula
\begin{equation*}
    \int_\Rd \phi(x) (-\Delta)^s \psi(x)\dx = a(\phi,\psi) \qquad \forall \phi, \psi \in C^\infty_c(\Rd).
\end{equation*}
Problem \eqref{eq:fracDirProb} can be written in weak form as follows:
\begin{equation} \label{eq:main_prob_weak}
    \text{Find } u \in V \text{ such that }
    a(u, v) = \int_{\Omega} f(x) v(x) \dx \quad \forall v \in V.
\end{equation}

\paragraph{Regularity of the domain and regularized distance functions.}

Given the distance function  $\dOmega(x) = \dist(x, \R^d\setminus \Omega)$ for $x\in \R^d$, we consider modified distance functions $\delta$ satisfying for some  $c, c_j > 0$:
\begin{equation}
    \label{as:delta}\tag{$\textup{A}_{\delta}^{\sigma}$}
    \begin{aligned}
         & \delta \in C^\infty(\Omega) \cap W^{\sigma,\infty}(\Omega) \textup{ with }
        \tfrac{1}{c} d_\Omega \leq \delta \leq c\, d_\Omega \textup{ in } \Rd \textup{ and}
        \\
         & \textup{$|D^j \delta| \leq c_j \delta^{\sigma - |j|}$ in  $\Omega$ for all multi-indices $j$ with $|j| > \sigma$}.
    \end{aligned}
\end{equation}
\begin{remark}
    The existence of a function $\delta$ satisfying \eqref{as:delta} can be ensured under regularity assumptions on the domain $\Omega$. More precisely,
    a natural choice in general domains can be found in Lemma A.2 in \cite{AbatangeloRos-Oton2020} under the assumption $\sigma>1$ and $\sigma \not\in \N$ and $\partial \Omega\in C^{k,t}$ with $k\in \N$, $t\in(0,1)$ and $k+t=\sigma$; with these hypotheses the function $\delta$ is shown to be obtained as the solution of
    \begin{align}\label{eq:distanceeq}
        \left\{
        \begin{aligned}
            -\Delta \delta(x) & = 1, &  & x \in \Omega,          \\
            \delta(x)         & = 0, &  & x \in \partial \Omega.
        \end{aligned}\right.
    \end{align}
    In some scenarios, there are simpler ways of obtaining a suitable function $\delta$ rather than solving \eqref{eq:distanceeq}. For example, if $\Omega=B_R(x_0)$, we can choose $\delta(x) = R^2 - |x-x_0|^2$ or $\delta(x)=R^4-|x-x_0|^4$. For reasons we will discuss below, we will often take the latter for the numerical experiments.
\end{remark}

\paragraph{Generic regularity of the solution to the fractional Dirichlet problem.}
Using super-solution arguments, it was shown (Lemma 2.6 \cite{RosOtonSerra2014}) that the solution $u$ of problem \eqref{eq:fracDirProb} satisfies the following estimate:
\begin{equation}
    \label{eq:u like ds if f bounded}
    |u(x)| \le C \|f\|_{L^\infty(\Omega)} \, \dOmega(x)^s , \qquad \forall f \in L^\infty (\Omega).
\end{equation}
Moreover, explicit estimates for the Green kernel (see \cite{ChenSong1998}) yield
\begin{equation}\label{eq:fpositive}
    u(x) \ge c \, \dOmega(x)^s \int_\Omega f(y) \, \dOmega(y)^s \, \dy, \qquad \forall f \ge 0.
\end{equation}
Thus, $u$ behaves like $d_{\Omega}^s$ for nonnegative and bounded $f$.
In \cite{RosOtonSerra2014}, it was proved that if \(f \in C^\alpha(\overline{\Omega})\) with small \(\alpha>0\), then
\(
u / \dOmega^s \in C^{\alpha + s}(\overline{\Omega}).
\)
Similar result for arbitrary $\alpha>0$ where later established.
In particular, Theorems 1.3 and 1.4 of \cite{AbatangeloRos-Oton2020} proved the following higher regularity result, valid for a regularized distance \(\delta\) satisfying \eqref{as:delta} with \(\sigma = \gamma + 1\):
\begin{equation}
    \label{eq:AbatangeloRosOton results}
    \partial \Omega \in C^{\gamma+1}, \ f \in C^{(\gamma -s)_+} (\overline{\Omega})
    \ \Longrightarrow \
    \frac{u}{\delta^s} \in W^{\gamma,\infty}(\Omega),  \qquad \forall \gamma > 0 \text{ such that } \gamma, \gamma \pm s \notin \mathbb N ,
\end{equation}
for all \(\gamma > 0\) such that \(\gamma, \gamma \pm s \notin \mathbb N\).
We also mention \cite{Grubb2015, Grubb2023} where it is proved that if $\Omega \in C^\infty$ and $\delta, f \in C^\infty(\overline{\Omega})$ then $u / \delta^s \in C^\infty(\overline{\Omega})$.

\begin{remark}[Explicit solution for \(f = 1\) and \(\Omega\) a ball]
    \label{rem:f=1 and Omega ball}
    An explicit solution to
    \eqref{eq:fracDirProb} when $\Omega = B_R(x_0)$ and $f=1$ (see \cite{Getoor1961, Bogdan2010}) is given by
    \begin{equation} \label{eq:explicitSol}
        u^{\ast}(x) = c_{d,s} \left(R^2 - |x-x_0|^2\right)^s, \qquad \text{where } c_{d,s} = \frac{\Gamma(d/2)}{2^{2s} \Gamma((d+2s)/2) \Gamma(1+s)}.
    \end{equation}
    This example illustrates the high regularity of \(u / \dOmega^s\).
    We will employ this explicit solution in our numerical experiments.
\end{remark}

\normalcolor

\paragraph{Over-triangulations.}
Given $\Omega$ with possibly smooth boundary
let us consider a sequence $(\triangulation)_{h > 0}$ where each $\triangulation$ is a set of simplices  (i.e., intervals if $d=1$, triangles if $d=2$, and tetrahedra if $d=3$) that meshes a slightly larger domain $\Omega_h$ and that satisfies the following properties:
\begin{itemize}
    \item $\overline \Omega \subset \bigcup_{K \in \triangulation} \overline K$
          and for $K \in \triangulation$ we have $\overline K \cap \overline \Omega \ne \emptyset$.
          We define
          \[
              \Omega_h \defeq \operatorname{interior} \left( \bigcup_{K \in \triangulation} \overline K \right);
          \]
    \item for $K \in \triangulation$, we denote  by $h_K \defeq \operatorname{diam}(K)$ and $\xi_K$ the diameter of the largest sphere inscribed in $K$. We assume that there exists a constant $\sigma>0$ such that $\max_{K \in \mathcal{T}_h} \frac{h_K}{\xi_k} \leq \sigma$. In this case, we say that the triangulation is \textit{shape regular};
    \item there are constants $c_1, c_2 > 0$ and there exists $h>0$, such that $c_1 h \leq h_K \leq c_2 h$ for any $K \in \mathcal{T}_h$. We say that the triangulation is \textit{uniform};
    \item any face of any simplex $K \in \triangulation$ is either a subset of the boundary $\partial \Omega$ or a face for another simplex $L \in \triangulation$. We say that such a triangulation is \textit{conforming}.
\end{itemize}
We call this set of simplices an ``over-triangulation'' of $\Omega$.
Given an over-triangulation $\triangulation$ we define
\begin{equation*}
    \plinear \triangulation
    \defeq
    \{ u \in C(\overline \Omega_h) :
    u|_K \text{ is linear for all } K \in \triangulation\}.
\end{equation*}
Notice, when $d > 1$, if $\Omega$ is regular then $\Omega \subsetneq \Omega_h$.
See \Cref{fig:overtriang}.
Let us also introduce the notation
\begin{equation*}
    \delta^s \plinear{\triangulation} \defeq \{ \delta^s \varphi : \varphi\in \plinear{\triangulation}\}.
\end{equation*}
\begin{remark}
    Since we consider smooth domains \(\Omega\) and piecewise linear basis functions defined over simplices, one must either over-mesh or under-mesh the domain. As the multiplication by \(\delta^s\) disregards any information outside \(\Omega\), it is technically more convenient to over-mesh.
\end{remark}

\begin{figure}[!ht]
    \centering
    \includegraphics[width = 0.25\textwidth]{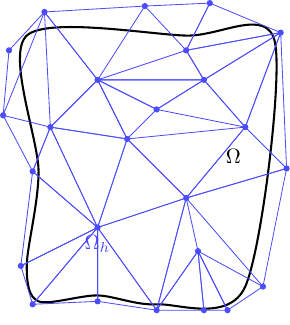}
    \caption{Example of over-triangulation.}
    \label{fig:overtriang}
\end{figure}

\paragraph{Weighted elements.} We want to set up a finite element method for problem \eqref{eq:main_prob_weak}. Given a regularization $\delta$ of the distance to the boundary function $d_\Omega$ (see \eqref{as:delta}),
and an over-triangulation $\mathcal T_h$, we consider the finite dimensional subspace of $H^{s}_0(\Omega)$ given by
\begin{equation} \label{eq:Vh}
    V_h \defeq \{v : \Rd \to \mathbb R \text{ such that } v |_\Omega \in \delta^s \plinear{\triangulation} \text{ and } v = 0 \text{ in } \Rd \setminus \Omega\} .
\end{equation}
Let us also define $N_h \coloneq \dim V_h$.
For the remaining elements, we drop the dependence on $h$ for the sake of the readability.
We define $\{ x_i \}_{i=1}^{N_h}$ be the vertices of all simplexes $K \in \triangulation$ (including those lying in $\partial \Omega_h$). The canonical basis of $\plinear \triangulation$ is given by the piece-wise linear function $\varphi_i$ such that $\varphi_i(x_j) = \delta_{ij}$. The canonical basis of $V_h$ is $\phi_i = \delta^s \varphi_i$.
When we use this basis, we will refer to the weighted Finite-Elements method, or WFEM for short.
\begin{remark}
    Notice $u / \delta^s$ may not vanish on $\partial \Omega$, so we need to include the nodes $x_i \in \partial \Omega_h$.
\end{remark}

The most common FEM approach for the Dirichlet problem considers on exactly-meshable domains $\Omega = \Omega_h$ (see, e.g., \cite{Borthagaray2017}).
In that case $\partial \Omega$ no better than Lipschitz, and this can be written in our notation precisely as $\Omega = \Omega_h$.
The standard approximation space is $V_h^{\text{FEM}} \defeq \plinear \triangulation \cap W_0^{1,\infty}(\Omega)$.
The canonical basis of this space is $\{\varphi_i : x_i \in \Omega\}$.

In order to visualize this difference, let us discuss the case $\Omega = (-1,1)$, and consider $h > 0$.
We denote the points $x_i = -1 + (i-1)h$ for $i = 1, \cdots, N_h$.
We can write the basis for standard FEM for the Dirichlet problem as
\begin{equation*}
    V_h^{\textup{FEM}} = \operatorname{span} \{\varphi_2, \cdots, \varphi_{N_h-1}\}.
\end{equation*}
The reader may see this as picking the basis $\phi_i^{\text{FEM}} = \varphi_{i+1}$ for $i = 1, \cdots, N_h - 2$.
The WFEM approximation space is
\begin{equation*}
    V_h = \operatorname{span} \{\delta^s \varphi_1, \cdots, \delta^s \varphi_{N_h}\}.
\end{equation*}
Notice that, can be seen as $\delta^s$ multiplying the basis for the classical Neumann problem.
We show the comparison in \Cref{fig: comparison profiles 1d}.
\begin{remark}
    Close to the boundary, we have that $\phi_1^{\textup{WFEM}}\sim \dOmega^s$ and $\phi_2^{\textup{WFEM}}\sim \dOmega^{1+s}$.
\end{remark}

\begin{figure}[t]
    \centering
    \begin{subfigure}{0.32 \textwidth}
        \centering
        \includegraphics[width = \textwidth]{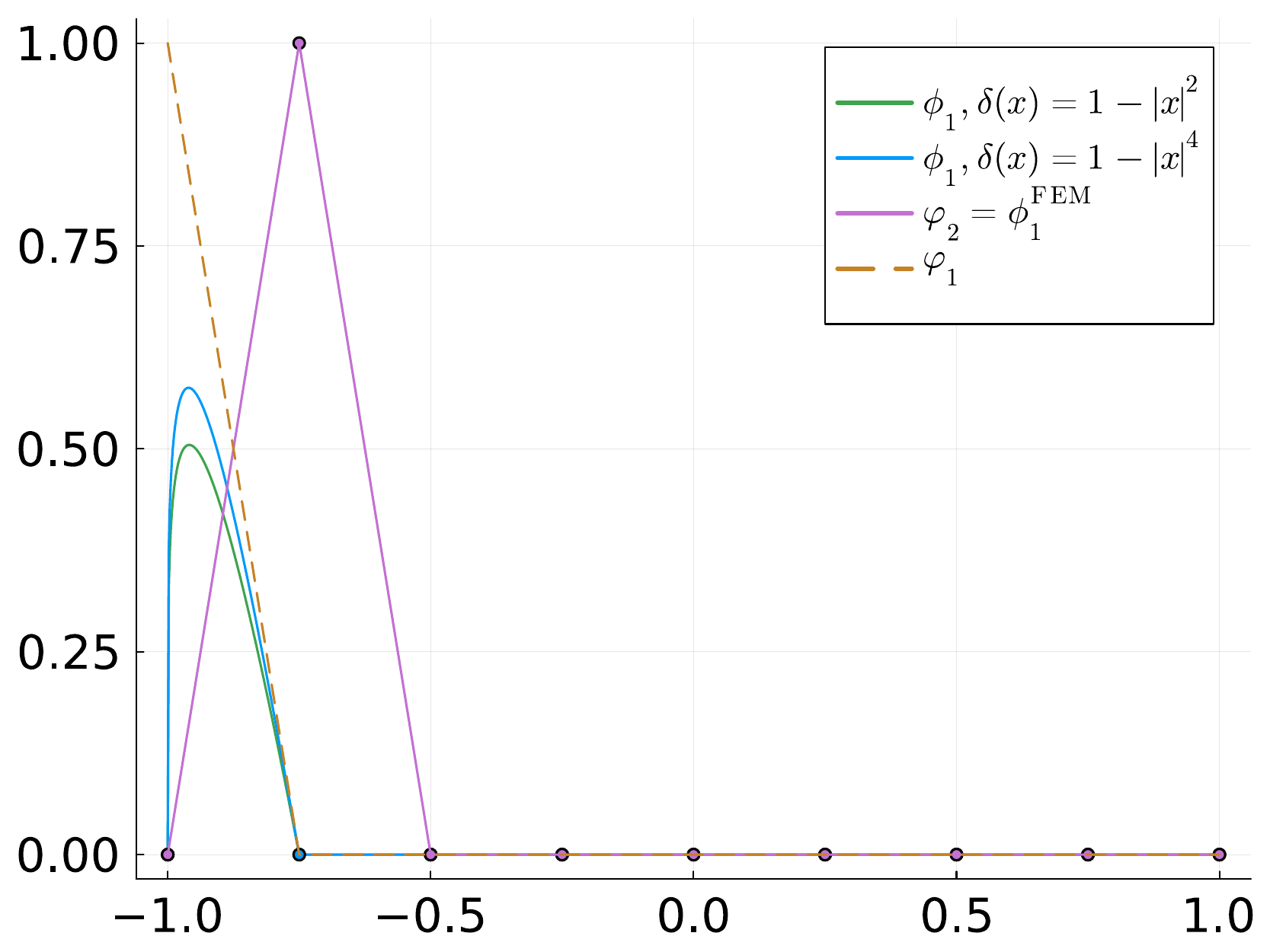}
        \caption{first basis element for FEM and WFEM}
    \end{subfigure}
    \hfill
    \begin{subfigure}{0.32\textwidth}
        \centering
        \includegraphics[width = \textwidth]{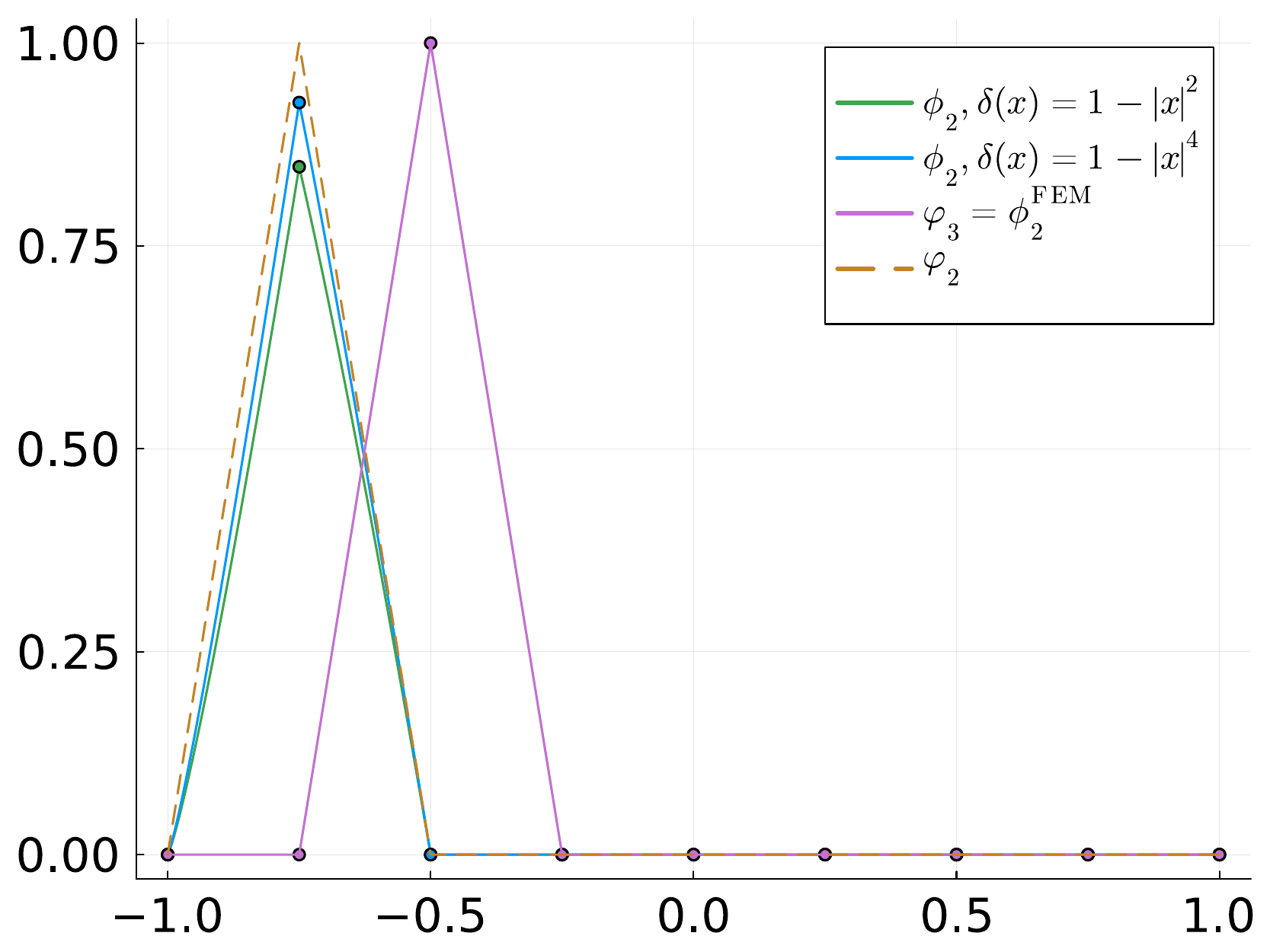}
        \caption{second basis element for FEM and WFEM.
        }
    \end{subfigure}
    \hfill
    \begin{subfigure}{0.32\textwidth}
        \centering
        \includegraphics[width = \textwidth]{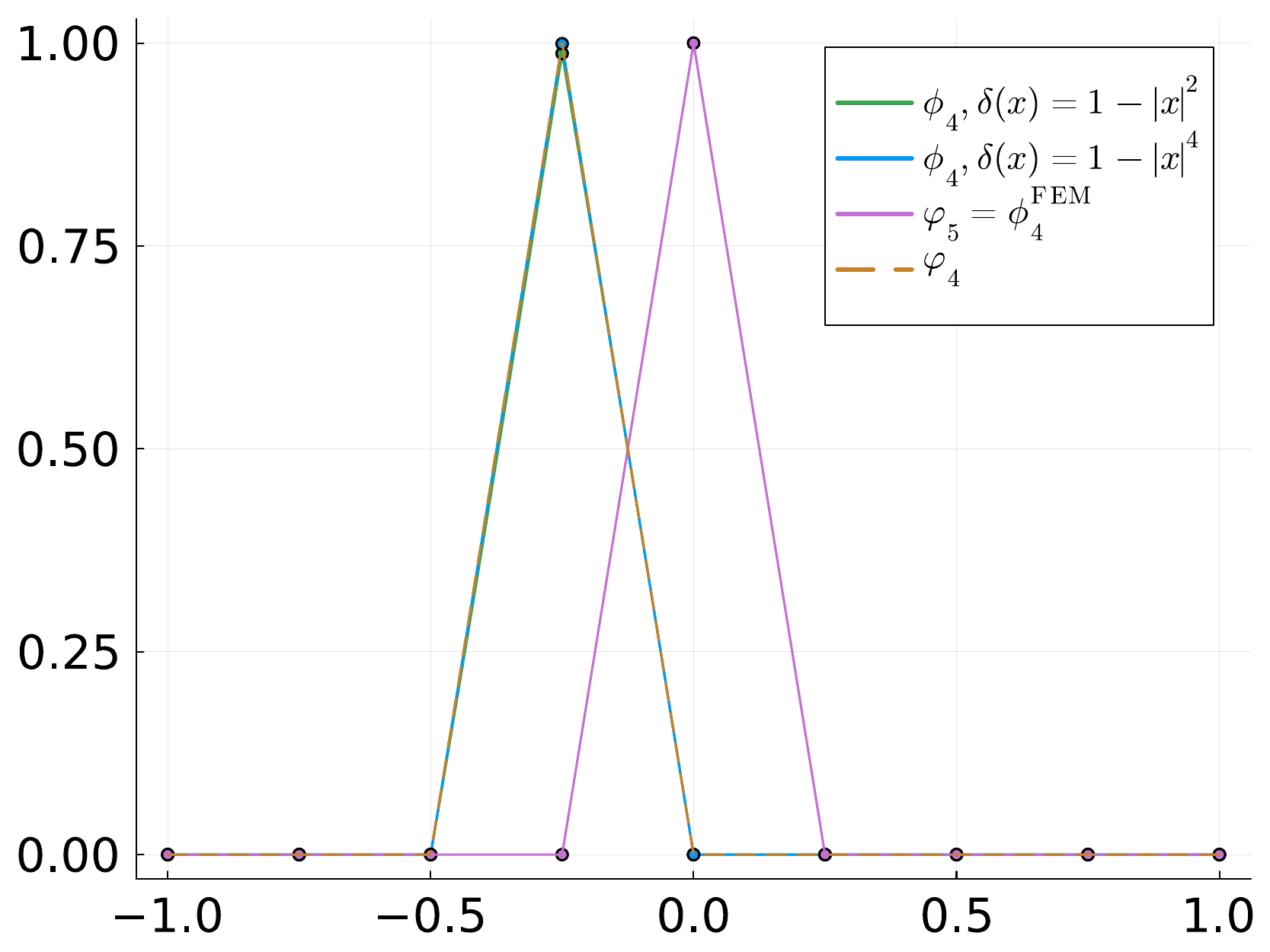}
        \caption{middle basis element for FEM and WFEM}
    \end{subfigure}
    \caption{FEM vs. WFEM basis function on $\Omega=(-1,1)$ for $s=0.2$ and $h=0.25$\normalcolor}
    \label{fig: comparison profiles 1d}
\end{figure}

In order to develop FEM theory, we show that $a(\cdot, \cdot)$ given by \eqref{eq:BinFormFracLap} is a continuous bilinear form on $V_h$.
The proof of the previous result will be presented later in \Cref{sec:proofs}.

\begin{lemma}\label{lem:subspace}
    Let $s\in(0,1)$, $\partial\Omega\in C^{1,1}$ and $\delta$ satisfy \eqref{as:delta} with $\sigma=1$.  Then $V_h \subset V$.
\end{lemma}

\paragraph{WFEM of the fractional Dirichlet problem.}
With this notation, we define the associated finite elements approximation problem as follows:
\begin{equation} \label{eq:main_prob_FE}
    \text{Find } u_h \in V_h \text{ such that }
    a(u_h, v_h) = \int_{\Omega} f(x) v_h(x) \dx \quad \forall v_h \in V_h.
\end{equation}
Since the bilinear form $a(\cdot, \cdot)$ is coercive
(see \eqref{eq:fractional Poincare}),
applying Lax-Milgram theorem yields well-posedness of \eqref{eq:main_prob_FE}.

\begin{remark}
    Due to \Cref{rem:f=1 and Omega ball}, if $\Omega = B_1$ and $f = 1$, and we pick the distance $\delta(x) = (1-|x|^2)_+$ we have precisely that $u / \delta^s$ is constant. This means that $u \in V_h$. Since $V_h \subset V$, then $u$ satisfies \eqref{eq:main_prob_FE}, and by uniqueness $u = u_h$. The numerical solution is exactly the \emph{true} solution.
\end{remark}

\subsection{Main results}

We present in this section the main results of our paper. They concern the convergence rate in $h$ for problem \eqref{eq:main_prob_FE} using weighted elements.  Firstly, in \Cref{thm:convergence rate Hs} we prove a general error estimate in $H^s(\mathbb{R}^d)$ seminorm, depending on the regularity of $u/\delta^s$ and the boundary of the domain $\partial \Omega$. Subsequently, we apply the result, under suitable regularity of the domain and the right-hand side $f$, to obtain \Cref{cor:convergence rate from regularity of f}.
Finally, in \Cref{cor:convergence rate in L2} we obtain an error estimate in the $L^2(\Omega)$ norm as well.

\subsubsection{Convergence rates in \texorpdfstring{$H^s$}{Hs}.}

We first formulate our convergence results in terms of the regularity of $u/\delta^s$.

\begin{theorem}
    \label{thm:convergence rate Hs}
    Let $s\in(0,1)$, $\partial \Omega\in C^{1,1}$, $\mu\in (s,2]$, and $\delta$ satisfy \eqref{as:delta} with $\sigma=\max\{1,\mu\}$.
    Let also $u$ be the corresponding solution of \eqref{eq:main_prob_weak}, assume that $u/\delta^s \in W^{\mu,\infty}({\Omega})$,
    and $V_h$ be given by \eqref{eq:Vh} and $u_h$ be the solution of \eqref{eq:main_prob_FE}. Then,
    we have that
    \begin{equation}\label{eq:Hs rate}
        [u - u_h]_{H^s (\R^d)} \le C h^{\mu-s}\lvert\log h\rvert^{\frac{1}{2}}  \|u/\delta^s\|_{W^{\mu,\infty}({\Omega})},
    \end{equation}
    where $C$ does not depend on $h$.
    In particular, the maximal order $h^{2-s}\lvert\log h\rvert^{\frac{1}{2}}$ is achieved if $\mu=2$, i.e., $u/\delta^s\in W^{2,\infty}(\Omega)$.
\end{theorem}

We can now apply the regularity result \eqref{eq:AbatangeloRosOton results} to obtain convergence rates in terms of the regularity of the right hand side $f$.

\begin{corollary}
    \label{cor:convergence rate from regularity of f}
    Let $s\in (0,1)$, $\gamma\in (s,3)$ such that $\gamma,\gamma\pm s\not\in\N $, $f\in W^{\gamma-s,\infty}(\Omega)$, $\partial \Omega\in C^{\gamma+1}\cap C^{1,1}$, and \eqref{as:delta} holds with $\sigma=\gamma+1$.
    Let $u$ be the corresponding solution of \eqref{eq:main_prob_weak},  $V_h$ be given by \eqref{eq:Vh}, and $u_h$ be the solution of \eqref{eq:main_prob_FE}.
    Then, we have that
    \begin{equation} \label{eq: estcorollary}
        [u - u_h]_{H^s (\R^d)} \le C h^{\min\{2-s,\gamma-s\}}\lvert\log h\rvert^{\frac{1}{2}}  \|f\|_{W^{\gamma-s,\infty}({\Omega})},
    \end{equation}
    where $C$ does not depend on $h$.
\end{corollary}

Notice that in the results above, the constants $C$ depend on all the parameters in the statement $s, \mu$ (or $\gamma$), the geometric parameters $\Omega, d$, the constants in \eqref{as:delta}, and the uniform constants in the definition of the triangulation $\triangulation$.

\normalcolor
\subsubsection{\texorpdfstring{$L^2$}{L2}-rates via Aubin--Nitsche duality}

Since our $H^s$-convergence requires high regularity of $u$ in the Hölder scale, we cannot directly apply the standard Aubin-Nitsche argument. We provide an adapted argument that produces rates we do not expect to be optimal. For simplicity of the presentation, we only formulate and prove the result using the most restrictive assumption on the regularity of the domain. A close inspection to the proof could lead a less restrictive assumption.

\begin{theorem}
    \label{cor:convergence rate in L2}
    Let the hypotheses of \Cref{cor:convergence rate from regularity of f} hold.
    Additionally assume that $\partial \Omega \in C^{2,1}$.
    Then we have that
    \begin{equation*}
        \|u - u_h\|_{L^2(\R^d)} \le C h^{\min\{2-s, \gamma -s\} + \alpha} \lvert\log h\rvert^{\frac 3 4}
        \|f\|_{W^{\gamma-s,\infty}({\Omega})}
        \quad
        \text{for almost all } \alpha \in \left(
        0,
        s \tfrac{4-2s}{4+d}\right),
    \end{equation*}
    where $C$ does not depend on $h$.
\end{theorem}

\subsection{Comments on related literature}

We briefly review previous finite element methods proposed for problem~\eqref{eq:fracDirProb}. The main challenges associated with this problem stem from the limited regularity of its solution and the singularity of the kernel in~\eqref{eq:FracLapInt}. Several recent works have addressed these issues, aiming to achieve both fast convergence and computational efficiency.

From a theoretical perspective, Acosta and Borthagaray first established in~\cite{Borthagaray2017} that for a Lipschitz domain $\Omega$, the piece-wise linear finite element approximation $u_h$ of the weak solution $u$ to problem~\eqref{eq:fracDirProb} satisfies the following estimate (see Theorem~4.6 therein):
\begin{align} \label{eq:Acosta-Borthagaray error}
    \| u  - u_h \|_{H^s (\mathbb{R}^d)} \le
    \begin{cases}
        C\, h^{1/2} \lvert\log h\rvert\, \| f \|_{C^{\frac{1}{2} -s} (\overline{\Omega})}, & \text{if } s \in (0,1/2),                        \\[4pt]
        C\, h^{1/2} \lvert\log h\rvert\, \| f \|_{L^\infty(\Omega)},                       & \text{if } s = 1/2,                              \\[4pt]
        C\, h^{1/2} \lvert\log h\rvert^{1/2}\, \| f \|_{C^\beta(\overline{\Omega})},       & \text{if } s \in (1/2,1) \text{ and } \beta > 0.
    \end{cases}
\end{align}
Our results indicate that the use of weighted basis  can substantially improve the convergence rate, achieving up to $ h^{2-s} \lvert\log h\rvert^{1/2}$, albeit at the cost of requiring higher regularity of both $f$ and $\Omega$.
Furthermore, Theorem~4.9 in \cite{Borthagaray2017} demonstrates that when graded meshes are employed, the convergence rate in~\eqref{eq:Acosta-Borthagaray error} can be improved up to an almost linear rate, proportional to $N^{-\frac 1 d} |\log N^{-\frac 1 d}|^{1/2}$, for $s \in (1/2,1)$ where in $N$ is the degrees of freedom (recall that in the case of uniform meshes $N_h \sim h^{-d}$).
The study of $H^1$ rates is done in \cite{BorthagarayCiarletJr2019}. In the same paper the author show for dimension $d=1$ that graded meshes allow for $H^s$-rates of order $h^{2-s}$ (see \cite[Remark 4.5]{BorthagarayCiarletJr2019}).
In the present work, we piece-wise linearly approximate $u/\delta^s$ which sufficiently smooth everywhere in $\Omega$ so graded meshes do not seem to provide a major improvement. This also simplifies the implementation.

Later, Bonito, Lei, and Pasciak~\cite{Bonito2019} improved the convergence rate to $h^{\kappa - s} \lvert\log h\rvert $ for $\kappa \in (s, 3/2)$, under the additional assumption that the solution to~\eqref{eq:fracDirProb} possesses higher regularity at boundary.
Their assumption on $u$, however, does not hold in general, for instance when $f$ is nonnegative and bounded (see \eqref{eq:u like ds if f bounded} and \eqref{eq:fpositive}).
Our approach is specifically designed to construct approximations for solutions of~\eqref{eq:fracDirProb} when $u$ has the generic behaviour $\dOmega^s$.
Nevertheless, in~\Cref{subsec:bonito} we also present numerical experiments for cases where $u$ is smooth, using a particular right-hand side function $f$ proposed in~\cite{Bonito2019}. Although we do not provide a theoretical justification for these results, the observed convergence rates appear to match, and in some instances even surpass those reported in~\cite{Bonito2019}.

Recently, \cite{Faustmann2022} established \emph{local interior} error estimates of order up to $h^{2-s}$ (see Corollary~2.4 therein), showing that the dominant errors associated with piece-wise linear basis functions are indeed concentrated near the boundary.
This observation further supports our approach, which is specifically designed to mitigate this primary source of error.
In \cite{Faustmann2023},
the authors apply $hp$-FEM (i.e. adaptive methods)
and get exponential convergence on the number of degrees of freedom.
The key disadvantage of this method is that it is notoriously difficult to implement.

\medskip
An alternative approach for this type of problem arises from the well-known Caffarelli--Silvestre extension~\cite{Caffarelli2007}, which connects the solution of the nonlocal problem~\eqref{eq:fracDirProb} in $\mathbb{R}^d$ with that of a corresponding local elliptic problem posed in $\mathbb{R}^{d+1}$.
Consequently, this approach is only valid for fractional powers of local operators.
We have not found in the literature FEM schemes for \eqref{eq:fracDirProb} using the Caffarelli--Silvestre extension.

\medskip
We want to emphasize that both our work and the previously cited contributions address problem~\eqref{eq:fracDirProb} involving the fractional Laplacian in its integral form, as defined in~\eqref{eq:FracLapInt} over the entire space~$\mathbb{R}^d$. However, the fractional Laplacian admits several distinct definitions (see, e.g.,~\cite{Kwasnicki, Whatisfractional}). Suitable localizations of these definitions to domains yield non-equivalent operators.  Among these alternative formulations, the \emph{spectral fractional Laplacian} has received particular attention due to its applications to problem~\eqref{eq:fracDirProb} and to other equations involving fractional diffusion operators. The solutions of the corresponding Dirichlet problem behave like $u \sim \dOmega$ (\cite{StingaCaffarelli2016}),
so the main difficulty of our framework is not present.
This operator also admits a Caffarelli-Silvestre type extension (\cite{StingaTorrea2010})
Finite element methods for the extension problem were first introduced in the pioneering work \cite{Nochetto2015}, achieving a convergence rate of $\lvert\log N\rvert^s N^{-1/(d+1)}$ in the $s$ fractional Sobolev norm corresponding to the problem, where $N$ stands for the degrees of freedom.
To deal with computational challenges arising from the extra dimension, see  \cite{Schwab2019}.
Subsequently, in \cite{Banjai2023}, the authors obtain exponential $H^s$ convergence rates using $hp$-FEM.
For other approaches see \cite{Bonito2015, Cusimano2018, Cusimano2020, Vabishchevich2015, Carrillo2025}.

\medskip

We conclude by mentioning that specialised finite element bases have been employed in other contexts, such as acoustic and electromagnetic scattering, to mitigate the growth in the number of degrees of freedom required by standard finite element methods to maintain a given accuracy. We refer the reader to \cite{Chandler2007, Chandler2012}, where products of piecewise polynomials with plane
waves were chosen, or to \cite{Groth2018}, for oscillatory basis functions capturing the high-frequency solution behaviour in a transmission problem.

\section{Proofs of the  main results}\label{sec:proofs}

In order to prove some of our results, we need to extend Hölder functions from $\Omega$ to $\Omega_h$ and $\R^d$. To this end, let $B_R$ be a ball such that
\[
    \overline{\Omega} \subset \overline{\Omega_h} \subset B_R .
\]
Let $\Omega$ be such that
$\partial \Omega \in C^{1,1}$. For any $\mu\in (0,2]$ consider a linear extension operator $E_\mu$ satisfying
\begin{equation}
    \label{eq:extension}
    \begin{aligned}
         & E_\mu: W^{\mu,\infty}(\overline{\Omega}) \to W^{\mu,\infty}(\R^d), \quad E_\mu u = u \ \text{in } \Omega, \quad \textup{supp}\{E_\mu u\} \subset B_R, \\
         & \|E_\mu u\|_{W^{\mu,\infty}(\R^d)} \leq C \, \|u\|_{W^{\mu,\infty}(\overline{\Omega})}, \quad \text{for some constant } C = C(d,\mu,\Omega,R).
    \end{aligned}
\end{equation}
The existence of such an extension operator can be found in \cite[Chapter VI, Theorem 3]{stein1970} and \cite[Lemma~6.37]{Gilbarg2001}.
We point out that in the construction of Stein $E_\mu = E_1$ for all $\mu \in (0,1]$.
Furthermore, $E_1 : C(\overline \Omega) \to C_b(\Rd)$.
Lastly, let us introduce the restriction $r_\Omega$ such that $r_\Omega u = u|_{\Omega}$.
We first make the observation that for $C^{1,1}$ domains the
power $s$ of the
regularized distance function $\delta$ belongs to $V$.
\begin{lemma}
    \label{lem:deltas in Hs0}
    Assume that $\partial \Omega \in C^{1,1}$ and that $\delta \in W^{1,\infty}(\Omega)$ satisfies $c \delta \le \dOmega \le C \delta$.
    Then, $\delta^s \in H^s(\Rd)$ for all $s\in (0,1)$.
\end{lemma}
\begin{proof}
    Let us decompose the integral taking advantage of the symmetry in $x$ and $y$
    \begin{equation*}
        \int_{\Rd}\int_{\Rd} \frac{|\delta(x)^s - \delta(y)^s|^2}{|x-y|^{d+2s}} \dy \dx =
        \int_{\Omega}\int_\Omega \frac{|\delta(x)^s - \delta(y)^s|^2}{|x-y|^{d+2s}} \dy \dx
        + 2 \int_{\Omega}\int_{\Rd \setminus \Omega} \frac{|\delta(x)^s - \delta(y)^s|^2}{|x-y|^{d+2s}} \dy \dx
    \end{equation*}
    If $x \in \Omega$ then we use that $\Rd \setminus \Omega \subset \Rd \setminus B(x, \dOmega(x))$ to estimate
    \begin{align*}
        \int_{\Rd \setminus \Omega} \frac{|\delta(x)^s - \delta(y)^s|^2}{|x-y|^{d+2s}} \dy & \le C \delta(x)^{2s} \int_{\dOmega(x)}^\infty r^{-d-2s} r^{d-1} \mathrm dr = C \delta(x)^{2s} \dOmega(x)^{-2s}.
    \end{align*}
    Therefore, we can integrate in $x$ to obtain that the integral in $\Omega \times (\Rd \setminus \Omega)$ is finite.
    In order to deal with the integral in $\Omega \times \Omega$ we point out that
    \begin{equation*}
        |\delta(x)^s - \delta(y)^s| \le \int_0^1 \left| \frac{ d  }{ d \theta} \delta ( \theta x + (1-\theta) y))^s \right| \mathrm{d} \theta \le C \min\{\delta(x), \delta(y)\}^{s-1} |x-y|.
    \end{equation*}
    We also have the estimate using Lipschitz regularity
    \begin{equation*}
        |\delta(x)^s - \delta(y)^s| \le C_s|\delta(x) - \delta(y)|^s \le C |x-y|^s.
    \end{equation*}
    Therefore, we have that
    \begin{align*}
        \int_{\Omega}\int_\Omega \frac{|\delta(x)^s - \delta(y)^s|^2}{|x-y|^{d+2s}} \dy \dx
         & \le C\int_{\Omega}\int_\Omega \min\{\delta(x), \delta(y)\}^{s-1} |x-y|^{-d-s+1} \dy \dx    \\
         & \leq C\int_{\Omega}\int_\Omega \min\{\dOmega(x), \dOmega(y)\}^{s-1} |x-y|^{-d-s+1} \dy \dx \\
         & = 2 C\int_\Omega \dOmega(x)^{s-1}
        \int_{\{ y \in \Omega : \dOmega(x) < \dOmega(y) \}}  |x-y|^{-d-s+1} \dy \dx.
    \end{align*}
    On the one hand, we can estimate
    \begin{equation*}
        \int_{\{ y \in \Omega : \dOmega(x) < \dOmega(y) \}} |x-y|^{-d-s+1} \dy \le \int_{B(x, \operatorname{diam}(\Omega))} |x-y|^{-d-s+1} \dy = C \operatorname{diam}(\Omega)^{1-s}
    \end{equation*}
    Since $\Omega \in C^{1,1}$
    we can use the tubular neighbourhood theorem to prove that
    \begin{equation*}
        \int_\Omega \dOmega (x)^{s-1} \dx < \infty.\qedhere
    \end{equation*}
\end{proof}
\normalcolor

We are now ready to prove \Cref{lem:subspace}.

\begin{proof}[Proof of \Cref{lem:subspace}]
    Let $v\in V_h$. By definition $v=0$ in $\R^d\setminus \Omega$. Let us now check that $v\in H^s(\R^d)$. Clearly,
    \[
        \|v\|_{L^2(\R^d)} = \|v\|_{L^2(\Omega)}\le|\Omega|\|v\|_{L^\infty(\R^d)}.
    \]
    On the other hand, we can write $v=\delta^s E_1 r_\Omega w$ for some $w\in \plinear\triangulation\subset W^{1,\infty}({\Omega}_h)$. Thus,
    \begin{align*}
        [v]_{H^s(\R^d)} & = [\delta^s E_1 w]_{H^s(\R^d)}                                                                            \\
                        & \leq \|\delta^s\|_{L^\infty(\R^d)} [ E_1 w]_{H^s(\R^d)}+[\delta^s ]_{H^s(\R^d)}\|E_1 w\|_{L^\infty(\R^d)}
    \end{align*}
    Clearly $\|\delta^s\|_{L^\infty(\R^d)},\|E_1 w\|_{L^\infty(\R^d)}<+\infty$ and due to  \Cref{lem:deltas in Hs0} we have $[\delta^s ]_{H^s(\R^d)} < \infty$. Finally, we note that, by the properties of the extension operator $E_1$ and the standard interpolation inequality (see \cite{Lunardi2018}), we have that
    \[
        \| E_1 w \|_{H^s(\R^d)}\leq C(s) \| E_1 w\|_{L^2(\R^d)}^{1-s}\| E_1 w\|_{H^1(\R^d)}^s \le C(s,R) \| E_1 w\|_{L^2(\R^d)}^{1-s}\| E_1 w\|_{W^{1,\infty}(\R^d)}^s<+\infty,
    \]
    where $R$ is the constant in \eqref{eq:extension}.
\end{proof}

\normalcolor

\subsection{Proof of convergence rates in \texorpdfstring{$H^s$}{Hs}}

As usual in the finite elements setting, we rely on the following Céa's lemma to prove the desired convergence result.

\begin{lemma}[Céa's lemma]\label{lem:Cea}
    Let $s\in(0,1)$, $\partial\Omega\in C^{1,1}$ and $\delta$ satisfy \eqref{as:delta} with $\sigma=1$.  Let also $u$ be the corresponding solution of \eqref{eq:main_prob_weak}, $V_h$ be given by \eqref{eq:Vh}, and $u_h$ be the solution of \eqref{eq:main_prob_FE}. Then,
    \begin{equation}
        \label{eq:Cea lemma}
        [u - u_h]_{H^s(\Rd)} = \min_{v_h \in V_h} [u - v_h]_{H^s(\Rd)}.
    \end{equation}
\end{lemma}
\begin{proof}
    We recall that $V_h \subset V$ as shown in Lemma \ref{lem:subspace}. The $\ge$ inequality in \eqref{eq:Cea lemma} holds because $u_h \in V_h$.
    Let $v_h \in V_h$.
    For any $w_h \in V_h$ we can use the weak formulation to deduce $a(u-u_h, w_h) = a(u , w_h) - a(u_h, w_h) = \int_\Omega f w_h - \int_\Omega f w_h = 0$.
    Therefore, we have that
    \begin{align*}
        [u - u_h]_{H^s(\Rd)}^2 & = a(u - u_h, u - u_h)
        = a(u - u_h, u) = a(u - u_h, u) - a(u - u_h, v_h)
        = a(u - u_h, u - v_h).
    \end{align*}
    Using the Cauchy-Schwarz inequality we have that
    \begin{align*}
        [u - u_h]_{H^s(\Rd)}^2 \le [u - u_h]_{H^s(\Rd)} [u - v_h]_{H^s(\Rd)}.
    \end{align*}
    If $[u - u_h]_{H^s(\Rd)} = 0$ then \eqref{eq:Cea lemma} holds trivially because $[\cdot]_{H^s(\Rd)} \ge 0$.
    If $[u - u_h]_{H^s(\Rd)} > 0$ then we can divide by $[u - u_h]_{H^s(\Rd)}$ to prove the $\le$ part of \eqref{eq:Cea lemma}.
\end{proof}

In order to build a suitable competitor to apply Céa's Lemma and obtain the desired orders of convergence, we need to define a suitable interpolant. More precisely, let $I_h : C(\R^d) \to PL(\triangulation)\subset W^{1,\infty}({\Omega}_h)$ be the piece-wise linear interpolation obtained by sampling on the vertexes of the triangles in $\triangulation$. We define, for $\mu\in (s,2]$  the map $J_h : \delta^s W^{\mu,\infty}({\R^d}) \to  V_h$  by
\begin{equation}\label{eq:competitor}
    J_h v\defeq \delta^s E_1
    r_\Omega
    I_h E_\mu \frac{v}{\delta^s}.
\end{equation}
We will not include the dependence on $\mu$ in the definition of $J_h$ for the sake of clarity.
Notice that since $E_\alpha w = w$ in $\Omega$ for any $\alpha$, for and $\alpha > 0$ we have that
\begin{equation*}
    J_h v = \delta^s E_\alpha
    r_\Omega
    I_h E_\mu \frac{v}{\delta^s}.
\end{equation*}
The aim of $J_h v$ is to give a rigorously meaning to $\delta^s I_h \frac{v}{\delta^s}$, where $v / \delta^s$ in only defined in $\Omega$ but not in $\Omega_h$.

We now show interpolators $I_h$ are bounded in $W^{\alpha,\infty}$ for $\alpha \in (0,1)$.
\begin{lemma}
    \label{lem:interpolant is uniformly bounded}
    Let $w \in W^{\alpha,\infty}(\Rd)$ for $\alpha \in (0,1]$.
    Then
    \begin{equation*}
        \|I_h w\|_{W^{\alpha,\infty}(\Omega)} \le C \|w\|_{W^{\alpha,\infty}(\Rd)} ,
    \end{equation*}
    where $C$ depends only on $\alpha$, $\Omega$, and the uniform constant of $\triangulation$.
\end{lemma}
\begin{proof}
    First,
    assume that $w \in W^{1,\infty}(\Rd)$.
    It is a trivial computation that for all $\phi \in W^{1,\infty}(\Rd)$
    \begin{equation}
        \label{eq:interpolation continuity Linfty and W1infty}
        \|I_h \phi \|_{L^\infty(\Omega_h)} \le \max_{\overline \Omega}|\phi|
        \qquad
        \|D I_h \phi \|_{L^\infty(\Omega_h)} \le C\max_{\overline \Omega_h}|D \phi|
    \end{equation}
    where $C$ does not depend on $h$.
    We observe that the operator $T_h = E_{1} \circ r_\Omega \circ I_h$ is
    \[
        T_h : C_b (\Rd) \to L^\infty(\Rd) \qquad T_h : C_b^1 (\Rd) \to W^{1,\infty}(\Rd).
    \]
    Applying the $K$-interpolation results in \cite[Theorem 1.6, Example 1.8, result after the proof Example 1.8]{Lunardi2018} that
    $
        T_h: W^{\alpha,\infty} (\Rd) \to W^{\alpha,\infty}(\Rd),
    $
    and we have
    \begin{equation*}
        \|T_h\|_{\mathcal L(W^{\alpha,\infty} (\Rd), W^{\alpha,\infty} (\Rd))}
        \le C(\alpha,d)
        \|T_h\|_{\mathcal L(C_b (\Rd), L^\infty (\Rd))}^{1-\alpha}
        \|T_h\|_{\mathcal L(C_b^1 (\Rd), W^{1,\infty} (\Rd))}^{\alpha}
    \end{equation*}
    and this can be bounded independently on $h$ by \eqref{eq:interpolation continuity Linfty and W1infty}.
    Therefore, we can write
    \begin{align*}
        \|I_h w\|_{W^{\alpha,\infty}(\Omega)} & \le \|E_1 r_\Omega I_h w \|_{W^{\alpha,\infty}(\Rd)} = \| T_h w \|_{W^{\alpha, \infty} (\Rd)} \le C \|w\|_{W^{\alpha,\infty}(\Rd)}. \qedhere
    \end{align*}
\end{proof}
We can show $W^{\alpha,\infty}$ rates of the interpolator $I_h$
\begin{lemma}
    \label{lem:interpolation}
    For $\alpha \in (0,1]$ and $ \mu \in (\alpha, 2]$ we have that for all $w \in W^{\mu,\infty}(\Rd)$
    \[
        \left\|w - I_h w\right\|_{W^{\alpha,\infty}
                (\Omega)
            } \le C h^{\mu-\alpha} \|w\|_{W^{\mu,\infty}(\Rd)},
    \]
    where $C$ depends only on $\alpha$, $\mu$, $\Omega$, and the uniform constants of the triangulation.
\end{lemma}
\begin{proof}
    First, we
    apply \cite[Theorem 3.1.5]{ciarletFiniteElementMethod2002} to show the result for some integer $\alpha, \mu$, i.e.,
    \begin{align*}
        \left\|w - I_h w \right\|_{L^\infty(\Omega_h)} \le C h^{2} \|w\|_{W^{2,\infty}(\Omega_h)},
        \qquad
        \left\|w - I_h w \right\|_{W^{1,\infty}(\Omega_h)} \le C h \|w\|_{W^{2,\infty}(\Omega_h)}.
    \end{align*}
    Let $\eta \in C^\infty_c(\Rd)$ with support $B_1$. We take $\eta_\sigma(x) = \sigma^{-d} \eta(x/\sigma)$.
    Let $w_\sigma = \eta_\sigma * w$. We have the estimates
    \begin{equation}
        \label{eq:convolution rates}
        \|w_\sigma\|_{W^{2,\infty}(\Rd)}
        \le C_\eta
        \sigma^{\mu - 2}
        \|w\|_{W^{\mu,\infty}(\Rd)},
        \qquad \|w - w_\sigma\|_{W^{\alpha,\infty}(\Rd)}
        \le C_\eta \sigma^{\mu-\alpha}\|w\|_{W^{\mu,\infty}(\Rd)}.
    \end{equation}
    Now we observe we that for $\alpha \in (0,1]$ we have
    \begin{equation}
        \label{eq:interpolation Cs}
        \begin{aligned}
            [u]_{W^{\alpha,\infty}(\Omega_h)}
             & = \sup_{\substack{x, y \in \Omega_h \\ x \ne y } }
            \frac{|u(x) - u(y)|}{|x-y|^\alpha}
            =
            \sup_{\substack{x, y \in \Omega_h      \\ x \ne y } }
            |u(x) - u(y)|^{1-\alpha} \left(\frac{|u(x) - u(y)|}{|x-y|} \right)^\alpha
            \\
             & \le
            2^{1-\alpha} \|u\|_{L^\infty(\Omega_h)}^{1-\alpha} [u]_{W^{1,\infty}(\Omega_h)}^\alpha.
        \end{aligned}
    \end{equation}
    Now we can use interpolation on the left-hand side applying \eqref{eq:interpolation Cs}
    to deduce that for $\alpha \in (0,1]$
    \begin{equation*}
        [ w - I_h w ]_{W^{\alpha,\infty}(\Omega_h)} \le C h^{2-\alpha} \|w\|_{W^{2,\infty}(\Omega_h)}.
    \end{equation*}
    As a last step, we interpolate on the right-hand side by using convolutions recalling the rates \eqref{eq:convolution rates}.
    We can write using \Cref{lem:interpolant is uniformly bounded} that
    \begin{align*}
        [w - I_h w]_{W^{\alpha,\infty}
                (\Omega)
            }
         & \le [w - w_\sigma]_{W^{\alpha,\infty}
                (\Omega)
            }
        + [w_\sigma - I_h w_\sigma]_{W^{\alpha,\infty}
                (\Omega)
            }
        + [I_h w_\sigma - I_hw ]_{W^{\alpha,\infty}
                (\Omega)
        }                                          \\
         & \le C [w - w_\sigma]_{W^{\alpha,\infty}
                (\Omega)
            }
        + [w_\sigma - I_h w_\sigma]_{W^{\alpha,\infty}
                (\Omega)
            }
    \end{align*}
    The first term on the right-hand side can be controlled using \eqref{eq:convolution rates} by $C \sigma^{\mu-\alpha}$.
    We estimate the other term also using \eqref{eq:convolution rates}
    \begin{equation*}
        [w_\sigma - I_h w_\sigma]_{W^{\alpha,\infty}
                (\Omega)
            } \le h^{2-\alpha} \|w_\sigma\|_{W^{2,\infty}(\Omega_h)} \le h^{2-\alpha} \sigma^{\mu - 2} \|w\|_{W^{\mu,\infty}(\Rd)}.
    \end{equation*}
    Taking $\sigma = h$ completes the proof.
\end{proof}
\normalcolor

Let us start prove now the order of approximation of a suitable competitor.

\begin{lemma}
    \label{lem:rate of Jh}
    Let $s\in(0,1)$, $\partial \Omega\in C^{1,1}$, $\mu\in (s,2]$, and $\delta$ satisfy \eqref{as:delta} with $\sigma=\max\{1,\mu\}$.
    Let also $v\in \delta^s C(\overline{\Omega})$ such that $v/\delta^s \in W^{\mu,\infty}(\overline{\Omega})$.
    Then, for all $\alpha\in (s,\min\{\mu,1\}]$, we have that
    \begin{equation}\label{eq:estcompeti1}
        [v - J_h v]_{H^s (\R^d)} \le \frac{C_1}{\sqrt{\alpha-s}} h^{\mu-\alpha}  \|v/\delta^s\|_{W^{\mu,\infty}({\Omega})},
    \end{equation}
    with $C_1=C_1(\Omega,d,s,\mu)$. In particular, we have that
    \begin{equation}\label{eq:estcompeti2}
        [v - J_h v]_{H^s (\R^d)} \le C_2 h^{\mu-s}\lvert\log h\rvert^{\frac{1}{2}}  \|v/\delta^s\|_{W^{\mu,\infty}({\Omega})},
    \end{equation}
    with $C_2=C_2(\Omega,d,s,\mu)$.
\end{lemma}

\begin{proof}
    Let us first note that we can write
    \[
        v-J_hv= \delta^s E_\alpha
        r_\Omega \normalcolor
        E_\mu \frac{v}{\delta^s}
        - \delta^s E_\alpha
        r_\Omega \normalcolor
        I_h E_\mu \frac{v}{\delta^s} = \delta^s E_\alpha
        r_\Omega \normalcolor
        \left(E_\mu\frac{v}{\delta^s} - I_h E_\mu \frac{v}{\delta^s}\right)
    \]
    Therefore, for all $\alpha \in(s,\min\{\mu,1\}]$, we can use Lemma \ref{lem:HsWalpha_inclusion} to get
    \begin{align*}
        [ v - J_h v ]_{H^s (\R^d)}
         & \le
        \|\delta^s\|_{L^\infty(\R^d)} \left[E_\alpha
            r_\Omega \normalcolor
            [E_\mu \frac{v}{\delta^s} - I_h E_\mu \frac{v}{\delta^s} ]\right]_{H^s (\R^d)}
        +
        [\delta^s]_{H^s(\R^d)} \left\| E_\alpha
        r_\Omega \normalcolor
        [ E_\mu\frac{v}{\delta^s} - I_h E_\mu \frac{v}{\delta^s} ] \right\|_{L^\infty (\R^d)}                                                                                                 \\
         & \lesssim \frac{\|\delta^s\|_{L^\infty(\R^d)} }{\sqrt{\alpha-s}} \left\|E_\alpha
        r_\Omega \normalcolor
        [E_\mu \frac{v}{\delta^s} - I_h E_\mu \frac{v}{\delta^s} ]\right\|_{W^{\alpha,\infty} (\R^d)}
        +
        [\delta^s]_{H^s(\R^d)} \left\|  E_\mu\frac{v}{\delta^s} - I_h E_\mu \frac{v}{\delta^s} \right\|_{
                W^{\alpha,\infty} (\Omega) \normalcolor
        }                                                                                                                                                                                     \\
         & \lesssim \frac{\|\delta^s\|_{L^\infty(\R^d)} }{\sqrt{\alpha-s}}\left\|E_\mu \frac{v}{\delta^s} - I_h E_\mu \frac{v}{\delta^s} \right\|_{W^{\alpha,\infty} (
                \Omega \normalcolor
                )}
        +
        [\delta^s]_{H^s(\R^d)} \left\|  E_\mu\frac{v}{\delta^s} - I_h E_\mu \frac{v}{\delta^s} \right\|_{
                W^{\alpha,\infty} (\Omega) \normalcolor
        }                                                                                                                                                                                     \\
         & \leq \frac{ \|\delta^s\|_{L^\infty(\R^d)} +[\delta^s]_{H^s(\R^d)}  }{\sqrt{\alpha-s}} \left\|E_\mu \frac{v}{\delta^s} - I_h E_\mu \frac{v}{\delta^s} \right\|_{W^{\alpha,\infty} (
                \Omega \normalcolor
                )},
    \end{align*}
    where we have used the symbol $\lesssim$ to express that the inequality holds up to a constant that depends only on $\Omega,d,s$ and $\mu$.
    Due to \Cref{lem:interpolation} we have that
    \[
        \left\|E_\mu \frac{v}{\delta^s} - I_h E_\mu \frac{v}{\delta^s} \right\|_{
                W^{\alpha,\infty} (\Omega)
                \normalcolor
            } \lesssim h^{\mu-\alpha} \left\|E_\mu \frac{v}{\delta^s}\right\|_{
                W^{\alpha,\infty} (\Omega)
                \normalcolor
            } \simeq h^{\mu-\alpha} \|\frac{v}{\delta^s}\|_{W^{\mu,\infty} ({\Omega})}.
    \]
    The proof of \eqref{eq:estcompeti1} is completed. For \eqref{eq:estcompeti2}, it is enough to take the sequence $\alpha_h=s+\frac{1}{2\lvert\log h\rvert}$ for $h$ small enough such that $\alpha_h\in (s,\min\{\mu,1\}]$, and use \eqref{eq:estcompeti1}.
\end{proof}
We are now ready to prove our main result of convergence.
\begin{proof}[Proof of  \Cref{thm:convergence rate Hs}]
    It is enough to combine \eqref{eq:Cea lemma} with the competitor estimate of  \eqref{eq:competitor} to get
    \begin{equation*}
        [ u - u_h ]_{H^s (\Rd)}  = \inf_{v_h \in V_h} [u - v_h ]_{H^s (\R^d)} \le [u - J_h u ]_{H^s(\R^d)},
    \end{equation*}
    and hence the desired error estimate \eqref{eq:Hs rate}.
\end{proof}

\begin{proof}[Proof of  \Cref{cor:convergence rate from regularity of f}]
    Let us first notice that, by Proposition 1.1 in \cite{RosOtonSerra2014} we have $u \in C^s(\mathbb{R}^d)$ and
    \begin{equation} \label{eq: Cs regularity}
        \| u \|_{W^{s,\infty}(\mathbb{R}^d)} \leq C_{\Omega, s} \| f \|_{L^\infty(\Omega)},
    \end{equation}
    where $C_{\Omega, s}$ is a positive constant depending only on the domain $\Omega$ and $s$.
    Applying the results of \cite{AbatangeloRos-Oton2020}, summarised in \eqref{eq:AbatangeloRosOton results}, we have $u/\delta^s \in W^{\gamma, \infty}(\Omega)$ with the inequality
    \begin{equation*}
        \| u/\delta^s\|_{W^{\gamma, \infty}(\Omega)} \leq C_{\Omega, d, s, \gamma} (\| f \|_{W^{\gamma - s, \infty}(\Omega)} + \| u\|_{L^\infty(\mathbb{R}^d)}),
    \end{equation*}
    for a positive constant $C_{\Omega, d, s, \gamma}$ depending only on $\Omega, d, s$ and $\gamma$. Inequality \eqref{eq: Cs regularity} therefore implies
    \begin{equation} \label{eq: Wgamma stability}
        \| u/\delta^s\|_{W^{\gamma, \infty}(\Omega)} \leq C_{\Omega, d, s, \gamma} \| f \|_{W^{\gamma - s, \infty}(\Omega)}.
    \end{equation}
    We can hence apply \Cref{thm:convergence rate Hs} with $\mu = \min \{ \gamma, 2 \}$, combining \eqref{eq: estcorollary} together with  \eqref{eq: Wgamma stability} to get the final result \eqref{eq: estcorollary}.
\end{proof}

\normalcolor

\subsection{Proof of convergence rates in \texorpdfstring{$L^2$}{L2}}
The standard approach to translate this $H^s$-rate to an improved $L^p$-rate is to use an Aubin-Nitsche duality argument.
The idea is to consider $e_h = u - u_h$ and note that for any $w_h \in V_h$ we have that
\[
    a(e_h, w_h) = a(u, w_h) - a(u_h, w_h) = \int_\Omega f(x)w_h(x)\dx - \int_\Omega f(x)w_h(x)\dx = 0.
\]
Thus, given any $\phi_h$, we consider the problem
\begin{equation}
    \label{eq:psih}
    (-\Delta)^s \psi_h = \phi_h \text{ in } \Omega,
    \qquad
    \psi_h = 0 \text{ in } \Rd \setminus \Omega,
\end{equation}
and note that for any $w_h \in V_h$, using the weak formulation of \eqref{eq:psih}, we have that
\begin{equation*}
    \begin{aligned}
        \int_\Omega e_h(x) \phi_h(x) \dx
         & = a(e_h, \psi_h) = a(e_h, \psi_h - w_h) \le [e_h]_{H^s(\R^d)} [\psi_h - w_h]_{H^s(\R^d)}.
    \end{aligned}
\end{equation*}
Therefore, we get the following result %
\begin{equation}
    \label{eq:Aubin-Nietsche key fact}
    \int_\Omega e_h(x) \phi_h(x) \dx \le [e_h]_{H^s(\R^d)} \inf_{w_h \in V_h} [\psi_h - w_h]_{H^s(\R^d)}.
\end{equation}
Due to Céa's lemma (\Cref{lem:Cea}) the infimum on the right-hand side is attained at the numerical solution of \eqref{eq:psih}.

\begin{remark}
    \label{rem:L2 rates}
    An interesting open question is whether \Cref{cor:convergence rate from regularity of f} could be improved for $f \in L^2(\Omega)$ to a result analogous to the classical case $s=1$, i.e.,
    \begin{equation*}
        [u - u_h]_{H^s(\mathbb{R}^d)} \le C h^{s} \lvert\log h\rvert^{\kappa} \|f\|_{L^2(\Omega)}.
    \end{equation*}
    In the spirit of \Cref{thm:convergence rate Hs}, this would mean that $u/\delta^s$ has ``$2s$ derivatives''.
    If this was the case we could prove, using $\phi_h = e_h$ and \eqref{eq:Aubin-Nietsche key fact}, an improved $L^2$-error estimate of order $h^2$, which is consistent with the rates suggested by the numerical experiments below.
\end{remark}

\begin{proof}[Proof of \Cref{cor:convergence rate in L2}]Let $e_h=u-u_h$. Using the Fourier transform $\mathcal F$, let us consider
    \begin{equation*}
        \phi_h \coloneqq \mathcal{F}^{-1}\!\big[(1 + |\xi|^2)^{-t}\, \mathcal{F}[e_h]\big].
    \end{equation*}
    This corresponds to applying the  operator $(I-\Delta)^{-t}$ in $\mathbb{R}^d$ to $e_h$. Let us also consider the following equivalent norm in $H^{r}(\mathbb{R}^d)$ for any $r \in \mathbb{R}$, given by
    \begin{equation*}
        \triplenorm{v}_{H^r(\mathbb{R}^d)}
        \coloneqq \|(1 + |\xi|^2)^{r/2}\, \mathcal{F}[v]\|_{L^2(\mathbb{R}^d)}.
    \end{equation*}
    We deduce, by Plancherel's theorem, that $\phi_h \in H^{s+2t}(\mathbb{R}^d)$ and $\triplenorm{\phi_h}_{H^{s+2t}(\mathbb{R}^d)}
        =  \triplenorm{e_h}_{H^s(\mathbb{R}^d)}$. Thus, by
    Morrey’s inequality, we have that $\phi_h \in W^{s+2t - \frac{d}{2},\infty}(\Omega)$
    provided that  $s + 2t - \tfrac{d}{2} > 0$. In particular
    \begin{equation*}
        \|\phi_h\|_{W^{s+2t - \frac{d}{2},\infty}(\Omega)}
        \leq C \triplenorm{\phi_h}_{H^{s+2t}(\mathbb{R}^d)}
        = C \triplenorm{e_h}_{H^s(\mathbb{R}^d)}, \quad \textup{for all } t>\frac{d-2s}{4}.
    \end{equation*}
    Let us now consider WFEM solution corresponding to \eqref{eq:psih}, that is
    \begin{equation*}
        \text{Find }  \Psi_h \in V_h \text{ such that } a(\Psi_h, v_h) = \int_\Omega \phi_h(x) v_h(x) \dx \text{ for all } v_h \in V_h.
    \end{equation*}
    We use now \Cref{cor:convergence rate from regularity of f} with $\overline{\gamma}=2s+2t-\frac{d}{2}$, provided $t$ is such that  $\overline{\gamma},\overline{\gamma}\pm s\not\in \N$, and $0<\overline{\gamma}-s<2-s$ to get
    \begin{equation*}
        \| \psi_h - \Psi_h \|_{H^s(\R^d)} \le C h^{\overline{\gamma}-s} \lvert\log h\rvert^{\frac 1 2} \|\phi_h\|_{W^{\overline{\gamma}-s,\infty}}
        \le C h^{\overline{\gamma}-s} \lvert\log h\rvert^{\frac 1 2} \|e_h\|_{H^s(\R^d)}
    \end{equation*}
    where have used that $\overline{\gamma}-s=s+2t-\frac d 2$. Note that we have used here the hypothesis $\partial \Omega \in C^{2,1}$.
    On the other hand, we can use Parceval's identity and proceed as in \eqref{eq:Aubin-Nietsche key fact} to conclude that
    \begin{align*}
        \triplenorm{e_h}_{H^{-t}(\R^d)}^2 & =\int_{\Rd} e_h(x) \phi_h(x)\dx =\int_\Omega e_h(x) \phi_h(x)\dx = a(e_h,\psi_h) = a(e_h,\psi_h-\Psi_h) \\
                                          & \leq  \|e_h\|_{H^{s}(\R^d)}\|\psi_h-\Psi_h\|_{H^{s}(\R^d)}                                              \\
                                          & \leq C h^{\overline{\gamma}-s} \lvert\log h\rvert^{\frac 1 2} \|e_h\|_{H^s(\R^d)}^2.
    \end{align*}
    Finally, let us note that with our chosen norms we have the following interpolation inequality if $-t \theta + s (1 - \theta) = 0$ (i.e. $\theta=\frac{s}{s+t}$):
    \begin{align*}
        \|e_h\|_{L^2(\Rd)} & \le \triplenorm{e_h}_{H^{-t}(\Rd)}^\theta \triplenorm{e_h}_{H^s(\Rd)}^{1-\theta}                           \\
                           & \leq C h^{\frac{\theta}{2}(\overline{\gamma}-s)} \lvert\log h\rvert^{\frac \theta 4} \|e_h\|_{H^s(\R^d)} .
    \end{align*}
    Summarizing, we have obtained that
    \[
        \|e_h\|_{L^2(\Rd)} \leq C h^{\frac{1}{2}\frac{s}{s+t}(\overline{\gamma}-s)} \lvert\log h\rvert^{\frac{1}{4}\frac{s}{s+t}} \|e_h\|_{H^s(\R^d)}
    \]
    with $\overline{\gamma}=2s+2t-\frac{d}{2}$ provided $t$ is such that  $\overline{\gamma},\overline{\gamma}\pm s\not\in \N$, and $0<\overline{\gamma}-s<2-s$. Note that, in particular, we have
    \[
        \|e_h\|_{L^2(\Rd)} \leq C h^{\frac{1}{2}\frac{s}{s+t}(s+2t-\frac{d}{2})} \lvert\log h\rvert^{\frac{1}{4}\frac{s}{s+t}} \|e_h\|_{H^s(\R^d)}
    \]
    for almost every $t\in \left(\frac{d-2s}{4},\frac{d-2s+2}{4} \right)$. This implies
    \[
        \|e_h\|_{L^2(\Rd)} \leq C h^{\alpha} \lvert\log h\rvert^{\frac{1}{4}} \|e_h\|_{H^s(\R^d)}
    \]
    for almost every $\alpha\in \left(0, s\frac{4-2s}{4+d} \right)$. Applying the estimate in \Cref{cor:convergence rate from regularity of f} we get the result.
\end{proof}

\section{Extensions and open problems}\label{sec:extOp}

In this section, we discuss several natural extensions of this work and highlight open questions whose resolution could improve the results presented in this paper.

\paragraph{Parabolic problems.} Some authors have studied higher regularity of $u$ and $u/\delta^s$ to the parabolic problem $\partial_t u + (-\Delta)^s u = f$ see \cite{FernandezRealRosOton2016,FernandezRealRosOton2017}.
Our approach has a standard generalisation to this context, see e.g., \cite{QuarteroniValli1994}.

\paragraph{Higher-order basis.}
Since the regularity of $u/\delta^s$ can be arbitrarily high (see \eqref{eq:AbatangeloRosOton results})
a possibility in these case would be to use higher order elements to achieve higher order convergence rates. A close inspection to the proofs reveals that for a basis of elements of order $k\in \mathbb{N}$, the order of convergence in the $s$-Sobolev norm can be $h^{k+1-s}$. \normalcolor

\paragraph{Generalised fractional Laplacians.}
The approach of this paper relies on the observation that solutions of \eqref{eq:fracDirProb} satisfy that the quotient $u/\delta^s$ is highly regular. This feature also appears in problems where the fractional Laplacian is replaced by more general nonlocal operators of the form
\[
    L\phi(y) = \textup{P.V.} \int_{\mathbb{R}^d} \big( \phi(x) - \phi(y) \big) K(x-y)\,\dx,
\]
under suitable structural conditions on the kernel $K$ (see \cite{AbatangeloRos-Oton2020,RosOton2016,RosOton2016_2,RosOtonSerra2017}). Consequently, weighted finite element methods, such as the one proposed here, are naturally extendable to this broader class of operators.

\paragraph{Singular kernels and related operators}
Another interesting family of problems is
\[
    (-\Delta)^s u(x) + V(x) u(x) = f(x),
\]
where the singularity of the potential $V$ at the boundary may alter the boundary behaviour of $u$.

On the one hand, for $s>1/2$ if we consider the so-called ``killing'' potential
\[
    \kappa(x) = C_{d,s} \int_{\Rd \setminus \Omega} |x-y|^{-d-2s} \dy \asymp d_\Omega (x)^{-2s}
\]
then $V(x) = -\kappa(x)$ gives the so-called Censored Fractional Laplacian $(-\Delta)_{\textup{CFL}}^s \defeq (-\Delta)^s + V$, which can equivalently be written as
\begin{equation*}
    (-\Delta)_{\textup{CFL}}^s u(x) = C_{d,s} \textup{P.V.} \int_\Omega  \frac{u(x) - u(y)}{|x-y|^{d+2s}}\dy.
\end{equation*}
In this setting it is known that for $f$ non-negative and bounded we have $u \asymp d_\Omega^{2s-1}$. The theory of regularity for $u / d_\Omega^{2s-1}$ is available in \cite{FallRosOton2021},
and the natural weighted basis for a finite elements theory like the one presented in this paper is $\delta^{2s-1} \plinear\triangulation$.

On the other hand, for any $\varepsilon \in (0,s)$ we can find $C(\varepsilon) > 0$ such that, if $V \ge C(\varepsilon) d_\Omega^{-2s}$, then $|u| \le C d_\Omega^{s+\varepsilon}$ near $\partial \Omega$ (see \cite{DiazGomezCastroVazquez2018}).
In this setting it would be interesting to see whether using the suitably weighted basis gives any improvements.

\paragraph{Semi-linear problems.}
Another interesting family of problems is $(-\Delta)^s u = f(u)$.
Depending on the behaviour of $f$ at $u = 0$ we have different behaviour at the boundary (see \cite{abatangeloLargesharmonicFunctions2015}).

\paragraph{Fractional $p$-Laplacian.}
The scheme proposed in this paper naturally extends to problems involving the fractional $p$-Laplacian,
\begin{equation*}
    (-\Delta)^s_p u(x) = C_{d,s,p}\,\operatorname{P.V.} \int_{\mathbb{R}^d} \frac{|u(x)-u(y)|^{p-2}(u(x)-u(y))}{|x-y|^{d+sp}}\dy.
\end{equation*}
Nevertheless, convergence rates cannot be established due to the absence of higher-order regularity results for $u/\delta^s$ comparable to those available for the case $p=2$.
Several authors have addressed regularity in this setting (see, e.g.,  \cite{BrascoLindgren2017,BrascoLindgrenSchikorra2018} and \cite{iannizzottoSurveyBoundaryRegularity2024} for a survey).
$C^\alpha$ regularity for $u/\delta^s$ can be found in \cite{IannizzottoMosconiSquassina2020,iannizzottoFineBoundaryRegularity2024}.
$C^{1,\alpha}$ regularity for local solutions $u$ with $f = 0$ has appeared as a preprint
\cite{GiovagnoliJesusSilvestre2025}.
We also note that a finite elements method with piece-wise linear basis was studied in \cite{BorthaLiNoc24,Otarola2024}.

\paragraph{Improved regularity gains.}

An interesting open question is whether the regularity gain in \eqref{eq:AbatangeloRosOton results} can be improved from $s$ to $2s$ derivatives. Such an improvement would yield the optimal $H^s$-convergence rate of order $2-s$ (see \Cref{cor:convergence rate from regularity of f}) under relaxed regularity requirements on the right-hand side $f$.

Another open question, discussed in \Cref{rem:L2 rates}, concerns the regularity of $u/\delta^s$ in the $H^\sigma$ scale for data $f \in L^2(\Omega)$. Note that, in this setting, the standard interpolant $I_h$ cannot be employed in high dimensions, and quasi-interpolants must be used instead.

\paragraph{Less regular domains.}
Throughout the paper we work with domains $\Omega$ that are, at least, $C^{1,1}$. This regularity is used to apply the Hölder regularity theory of $u/\delta^s$, to prove $\delta^s \in V$ (in particular to show $\delta^{s-1} \in L^1(\Omega)$), and for the existence suitable extension operator in $W^{1,\infty}(\Omega)$.
The latter technical facts probably hold under less regularity of $\Omega$, and hence \Cref{thm:convergence rate Hs} can be sharpened. However, to weaken the regularity assumption of $\Omega$ in \Cref{cor:convergence rate from regularity of f} would require improving the regularity theory of $u/\delta^s$.

\paragraph{Céa's lemma in Hölder spaces.} The results in this paper rely on $C^{\alpha}$ regularity estimates to derive $H^s$-type convergence rates. This approach appears to be suboptimal, as it requires passing from $C^{\alpha}$ to $H^{\alpha-\varepsilon}$ (and vice versa via Morrey’s embedding) incurring a loss of regularity.

It would be of interest to establish a Céa-type lemma in Hölder spaces $C^s(\Omega)$. Such a result would provide a natural framework for $C^s$ and $L^\infty$-error estimates. This kind of result is known in the classical case $s=1$ (see \cite{LarsScharleSuli2021}).

\normalcolor

\section{Numerical Experiments}

We now present numerical experiments in one dimension.
In order to compute the entries of the stiffness matrix $A_{i,j} = a(\phi_j, \phi_i)$ we have used a quadrature of the singular integral, that we will discuss in an upcoming work \cite{DelTesoCastroFronzoni2025}.
The experiments have been implemented in \texttt{julia} 1.11.7, and
the codes are available at
\begin{quote}
    \url{https://github.com/dgomezcastro/FEMFractionalQuadrature.jl}.
\end{quote}
Other approaches to approximate the stiffness matrix can be found in \cite{Bonito2019, Sheng2025}.

\subsection{Boundary accuracy of the interpolator \texorpdfstring{$I_h$}{Ih} and \texorpdfstring{$J_h$}{Jh}}

Let $u^*$ be the explicit solution in \eqref{eq:explicitSol} with $d=1$, $R=1$, and $x_0 = 0$, i.e., $\Omega = (-1,1)$ and $f = 1$.
In this case we can take $\Omega_h = \Omega$.
In \Cref{fig: interpolation 1d}, we can first qualitatively observe the improvement in the interpolation between $I_h u^*$ (piece-wise linear interpolation) and $J_h u^*$ (our tailored interpolation) for different choices of the regularised version of $d_\Omega$, denoted by $\delta$ throughout the paper, and a fixed $h$.

Although the overall solutions appear similar in
Figures~\ref{fig: interpolation 1d}%
\subref{fig:s=0.1 plot solution whole domain}
to~\subref{fig:s=0.6 plot solution whole domain},
a closer inspection near the boundary (see Figures \ref{fig: interpolation 1d}%
\subref{fig:s=0.1 plot solution zoom}
to~\subref{fig:s=0.6 plot solution zoom}) clearly reveals the remarkable capability of WFEM in capturing the behaviour of the exact solution.
In particular, we can observe how the solution obtained with WFEM closely matches the profile of $u^*$ near the boundary.
Since the $J_h u^*$ associated to $\delta(x) = 1-|x|^2$ and $\delta(x)=1-|x|^4$ seem equivalent in
Figures~\ref{fig: interpolation 1d}\subref{fig:s=0.1 plot solution zoom}
to~\subref{fig:s=0.6 plot solution zoom}, we provide
Figures~\ref{fig: interpolation 1d}\subref{fig:s=0.1 plot I_h(u/delta^s)}
to~\subref{fig:s=0.6 plot I_h(u/delta^s)} where we show $I_h(u^*/\delta^s) = J_h(u^*) / \delta^s$.
\normalcolor

\begin{figure}[p]
    \centering

    \begin{subfigure}{0.32 \textwidth}
        \centering
        \includegraphics[width = \textwidth]{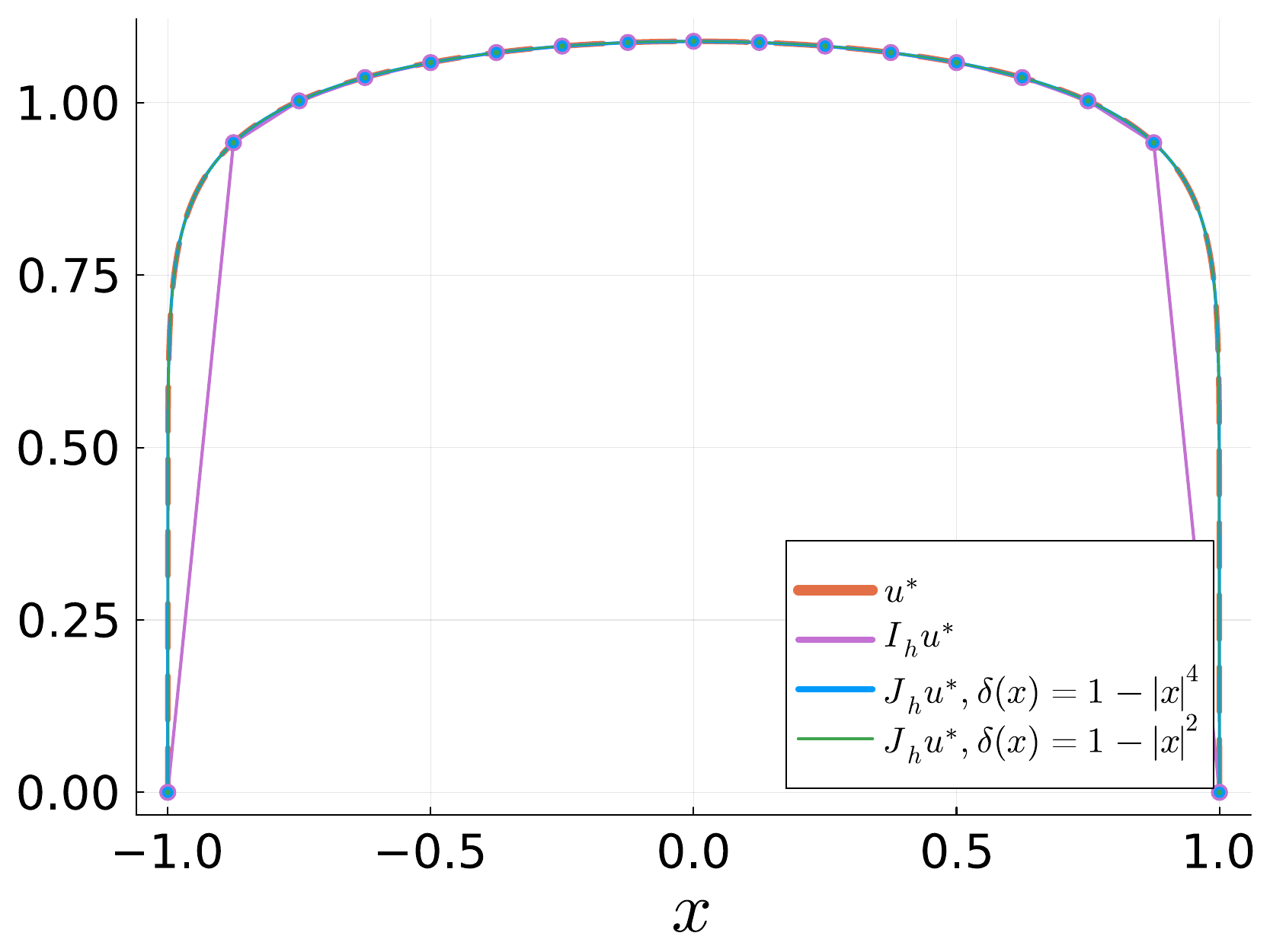}
        \caption{ $s=0.1$  }
        \label{fig:s=0.1 plot solution whole domain}
    \end{subfigure}
    \hfill
    \begin{subfigure}{0.32 \textwidth}
        \centering
        \includegraphics[width = \textwidth]{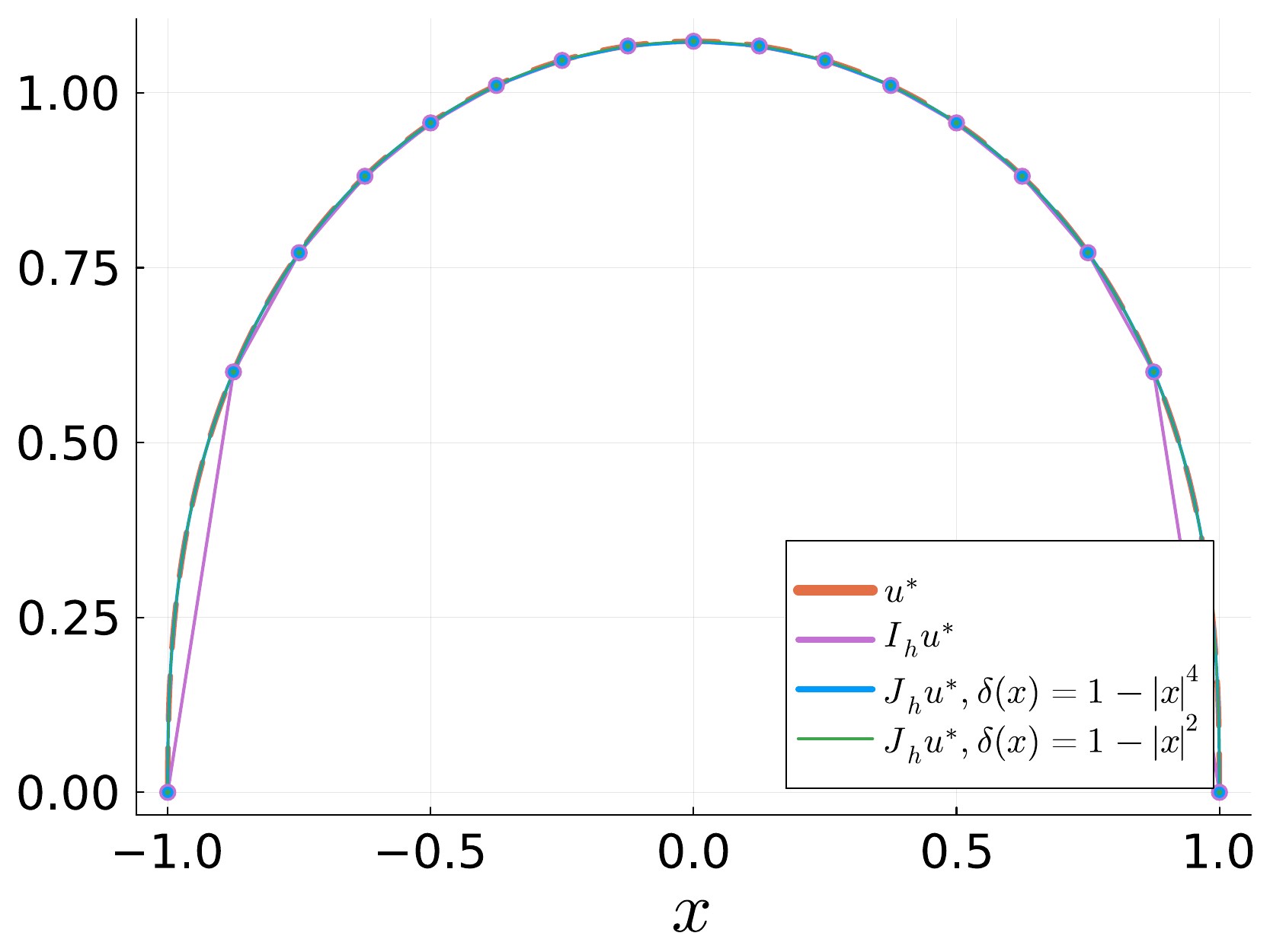}
        \caption{ $s=0.4$}
        \label{fig:s=0.4 plot solution whole domain}
    \end{subfigure}
    \hfill
    \begin{subfigure}{0.32 \textwidth}
        \centering
        \includegraphics[width = \textwidth]{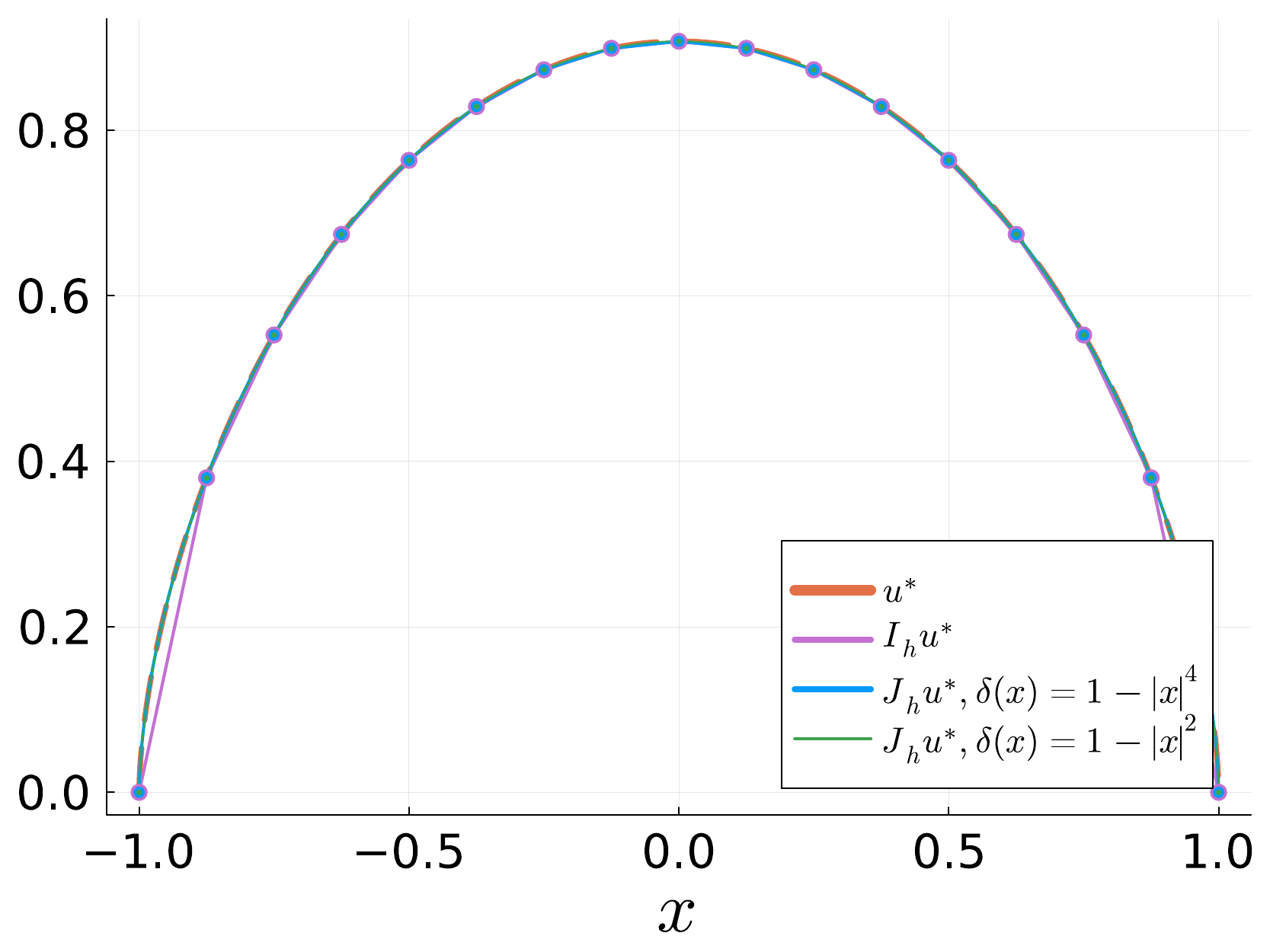}
        \caption{ $s=0.6$}
        \label{fig:s=0.6 plot solution whole domain}
    \end{subfigure}
    \hfill
    \begin{subfigure}{0.32 \textwidth}
        \centering
        \includegraphics[width = \textwidth]{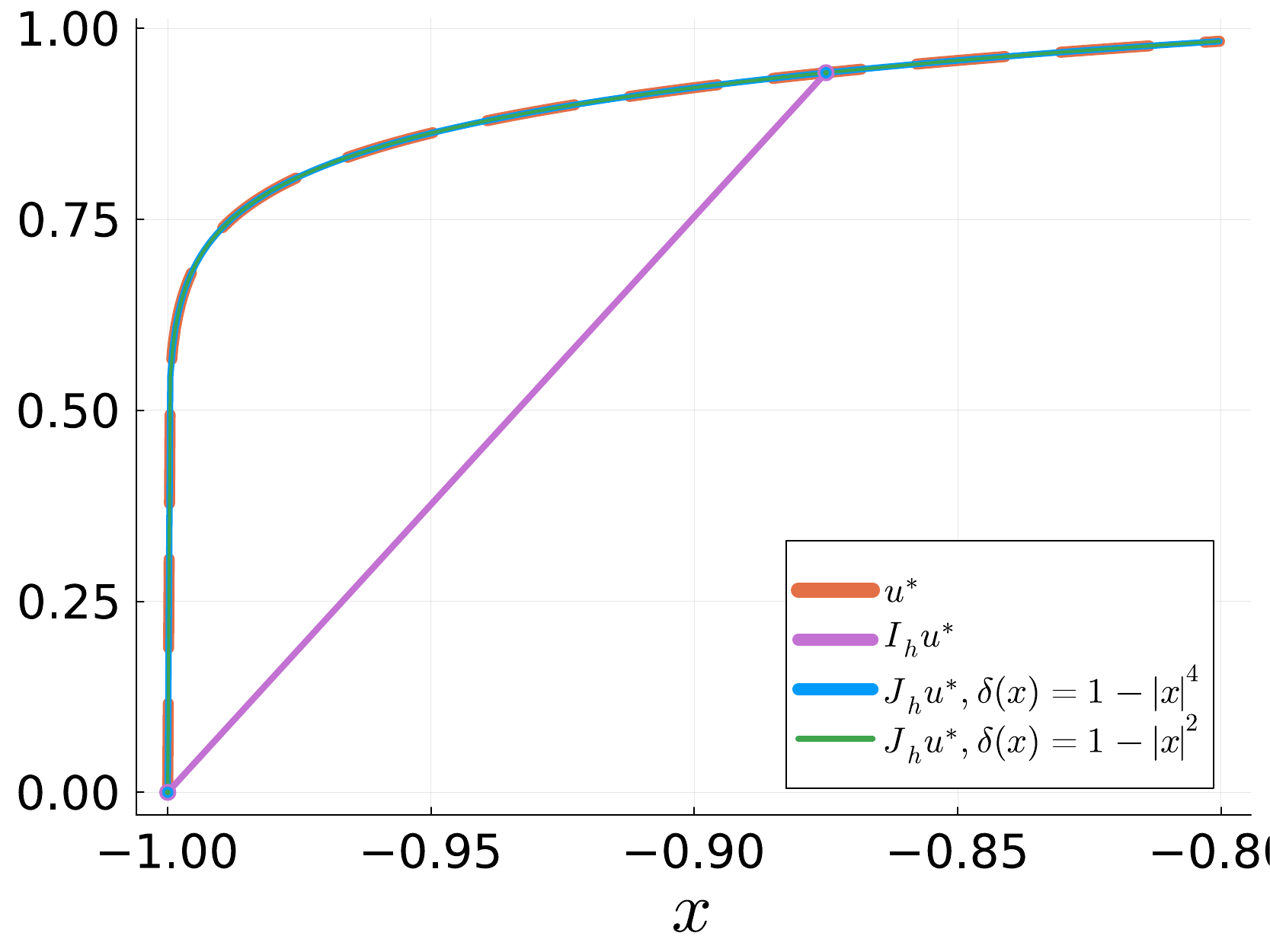}
        \caption{ $s=0.1$, zoom at the boundary }
        \label{fig:s=0.1 plot solution zoom}
    \end{subfigure}
    \hfill
    \begin{subfigure}{0.32 \textwidth}
        \centering
        \includegraphics[width = \textwidth]{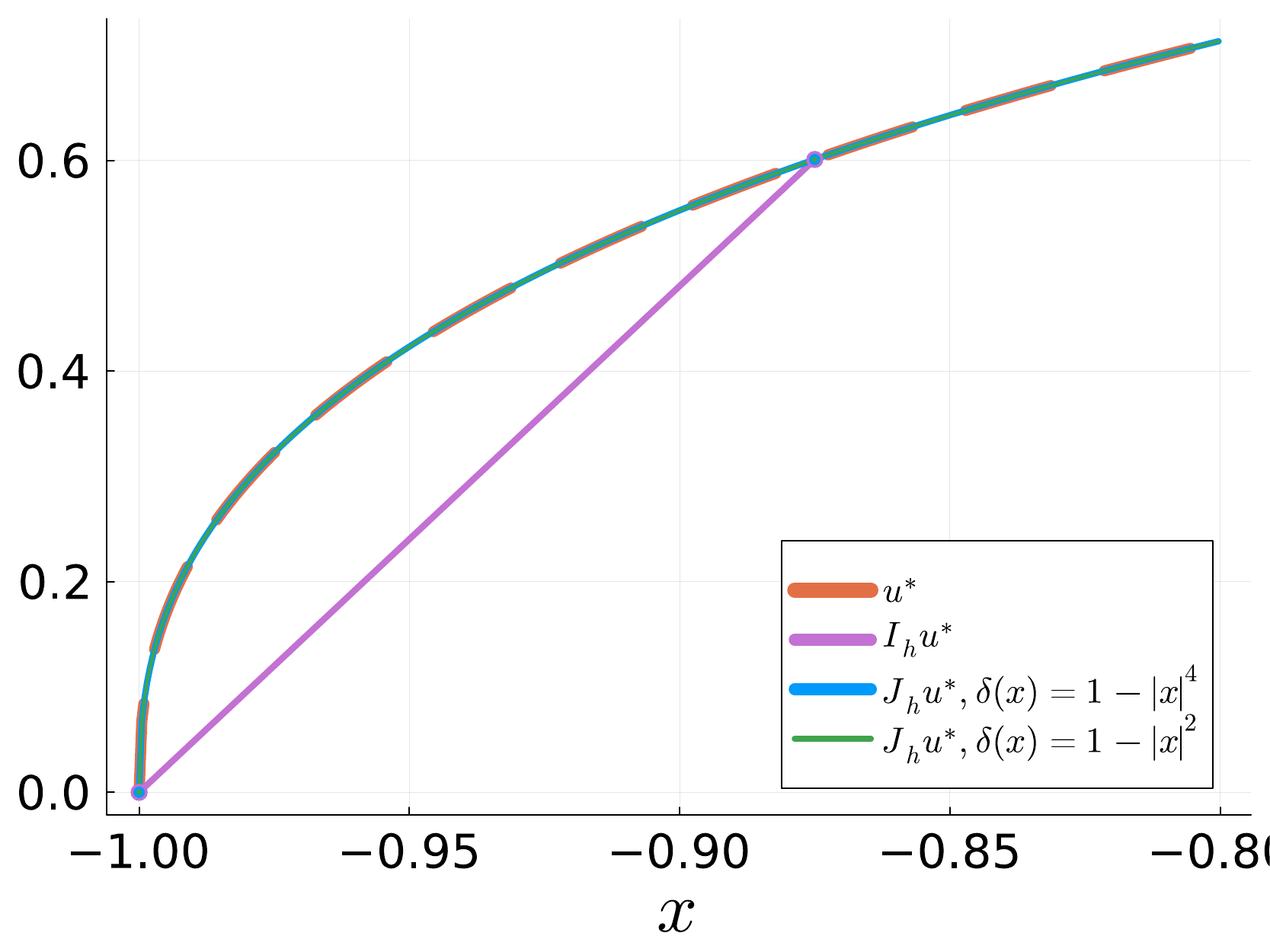}
        \caption{  $s=0.4$, zoom at the boundary }
        \label{fig:s=0.4 plot solution zoom}
    \end{subfigure}
    \hfill
    \begin{subfigure}{0.32 \textwidth}
        \centering
        \includegraphics[width = \textwidth]{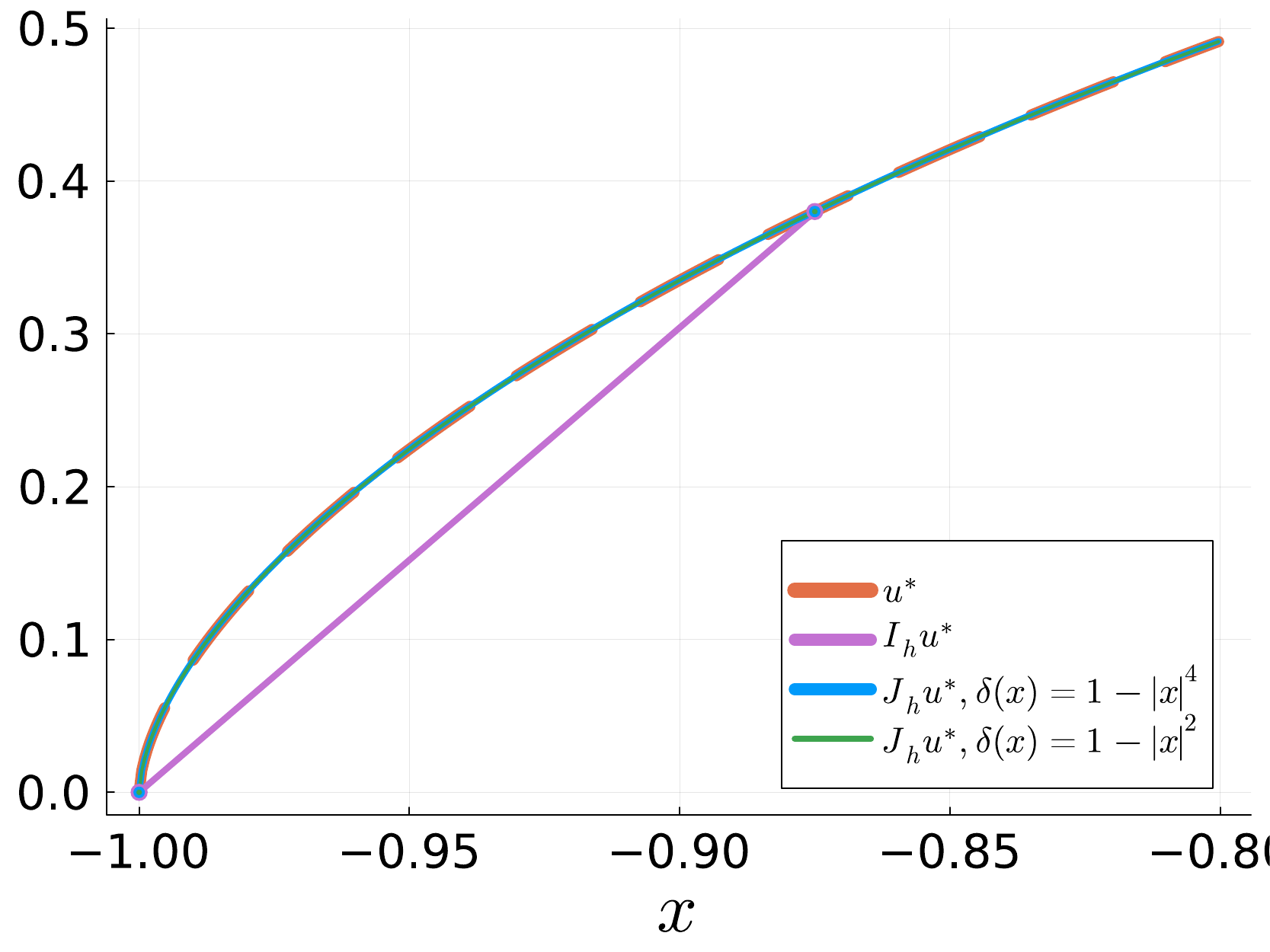}
        \caption{  $s=0.6$, zoom at the boundary }
        \label{fig:s=0.6 plot solution zoom}
    \end{subfigure}
    \hfill
    \begin{subfigure}{0.32 \textwidth}
        \centering
        \includegraphics[width = \textwidth]{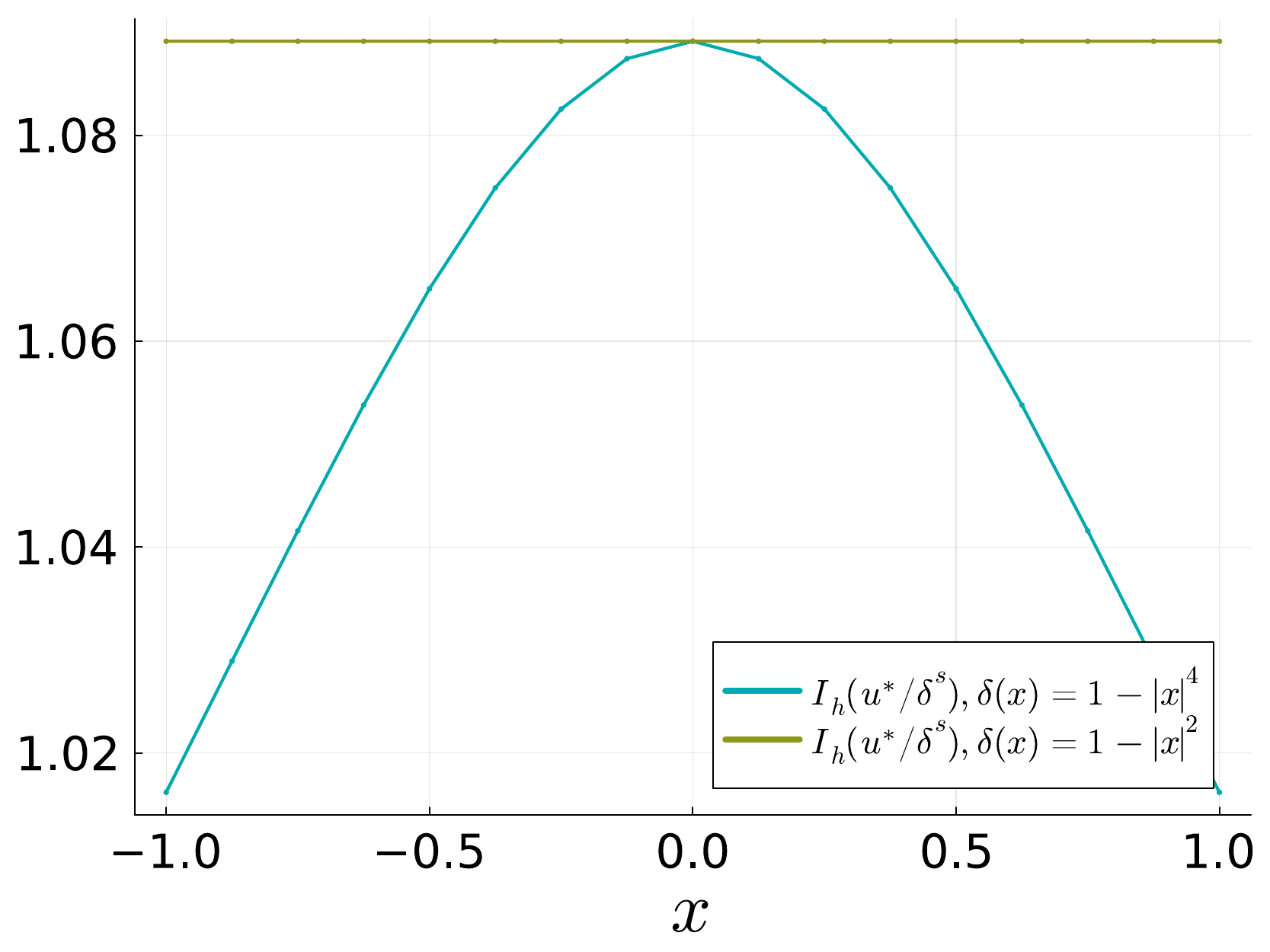}
        \caption{ $s=0.1$}
        \label{fig:s=0.1 plot I_h(u/delta^s)}
    \end{subfigure}
    \hfill
    \begin{subfigure}{0.32 \textwidth}
        \centering
        \includegraphics[width = \textwidth]{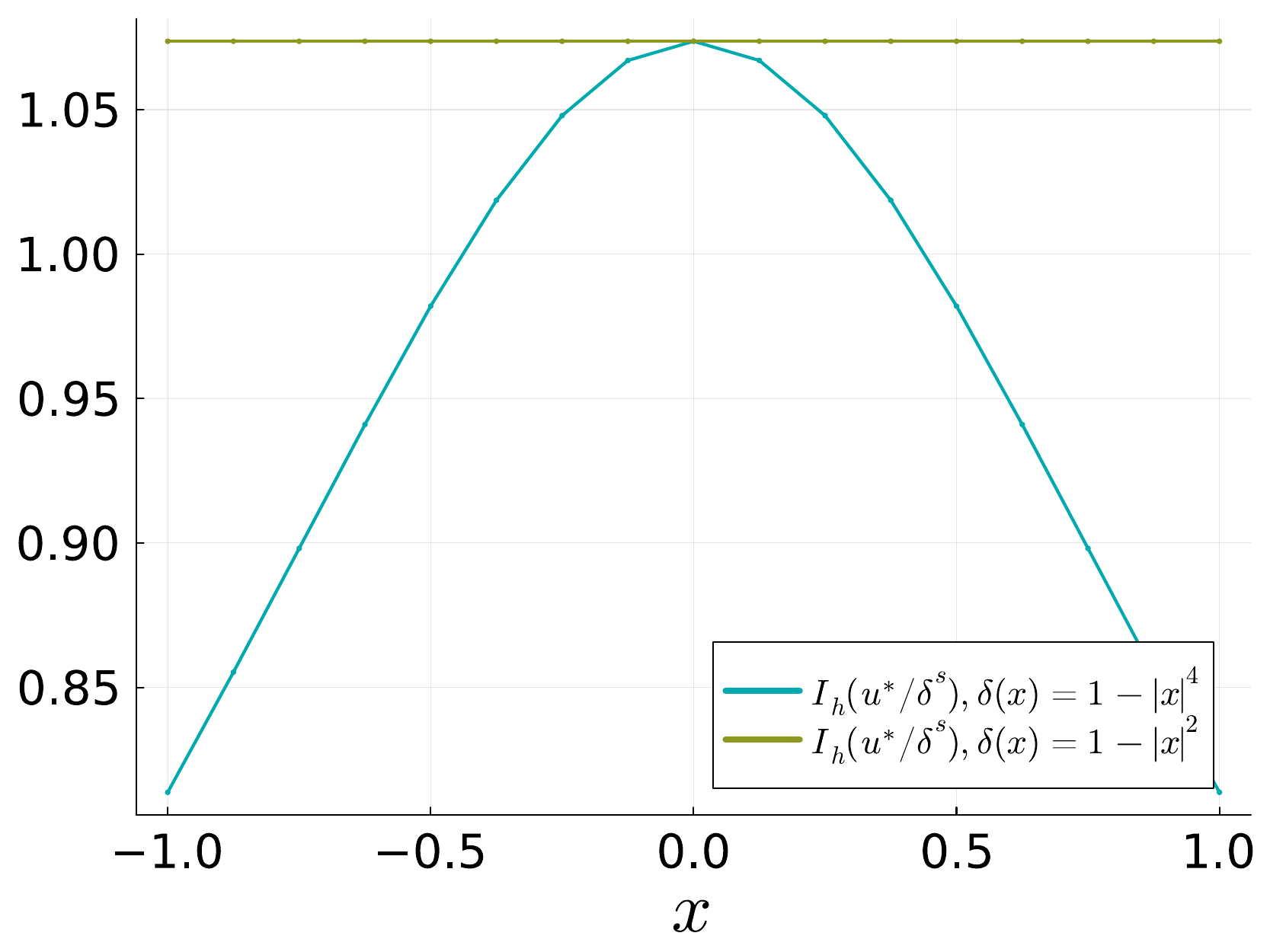}
        \caption{  $s=0.4$}
        \label{fig:s=0.4 plot I_h(u/delta^s)}
    \end{subfigure}
    \hfill
    \begin{subfigure}{0.32 \textwidth}
        \centering
        \includegraphics[width = \textwidth]{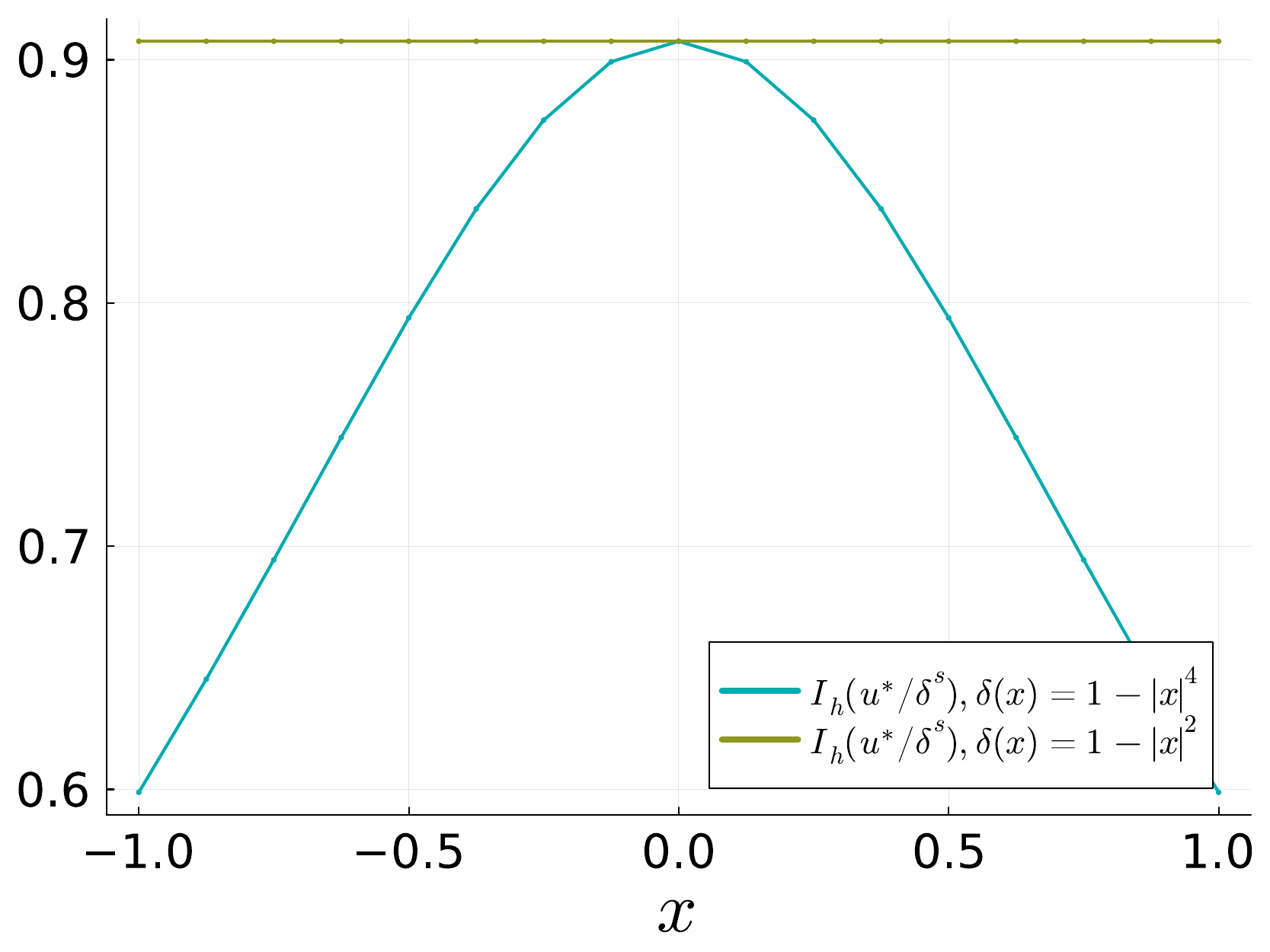}
        \caption{  $s=0.6$}
        \label{fig:s=0.6 plot I_h(u/delta^s)}
    \end{subfigure}
    \caption{Comparison between explicit solution
        \eqref{eq:explicitSol}, piece-wise linear interpolation $I_h u^\ast$ and $J_h u^\ast$, defined in \eqref{eq:competitor}, for different choices of the function $\delta$. $\Omega=\Omega_h=(-1,1)$, uniform meshes with $h=|\Omega|/2^4$.
    }
    \label{fig: interpolation 1d}
\end{figure}

\subsection{\texorpdfstring{$H^s$}{Hs} and \texorpdfstring{$L^2$}{L2} rates when \texorpdfstring{$f = 1$}{f=1}}
As before, let $u^*$ be the explicit solution in \eqref{eq:explicitSol} with $d=1$, $R=1$, and $x_0 = 0$, i.e., $\Omega = (-1,1)$ and $f = 1$.
In \Cref{fig: convergence Hs L2}  we validate the theoretical results that we have presented in
\Cref{thm:convergence rate Hs,cor:convergence rate from regularity of f,cor:convergence rate in L2}.
In particular, with the choice of $f=1$ we put ourselves in the case of maximal $H^s$-rate of \Cref{cor:convergence rate from regularity of f}, i.e., $h^{2-s}\lvert\log h\rvert^{\frac{1}{2}}$.
As shown in the \Cref{fig: convergence Hs L2}\subref{fig: convergence Hs f=1} the observed convergence rates match the theoretical prediction, with the observed slope of every convergence line in log scale for different $s$ being approximately $2-s$. \Cref{fig: convergence Hs L2}\subref{fig: convergence L2 f=1} instead is linked to the fact that, even we are only able to prove \Cref{cor:convergence rate in L2}, we would expect the convergence to be quadratic in $L^2(\Omega)$.

\begin{figure}[!ht]
    \centering

    \begin{subfigure}{0.49 \textwidth}
        \centering
        \includegraphics[width=0.99\textwidth]{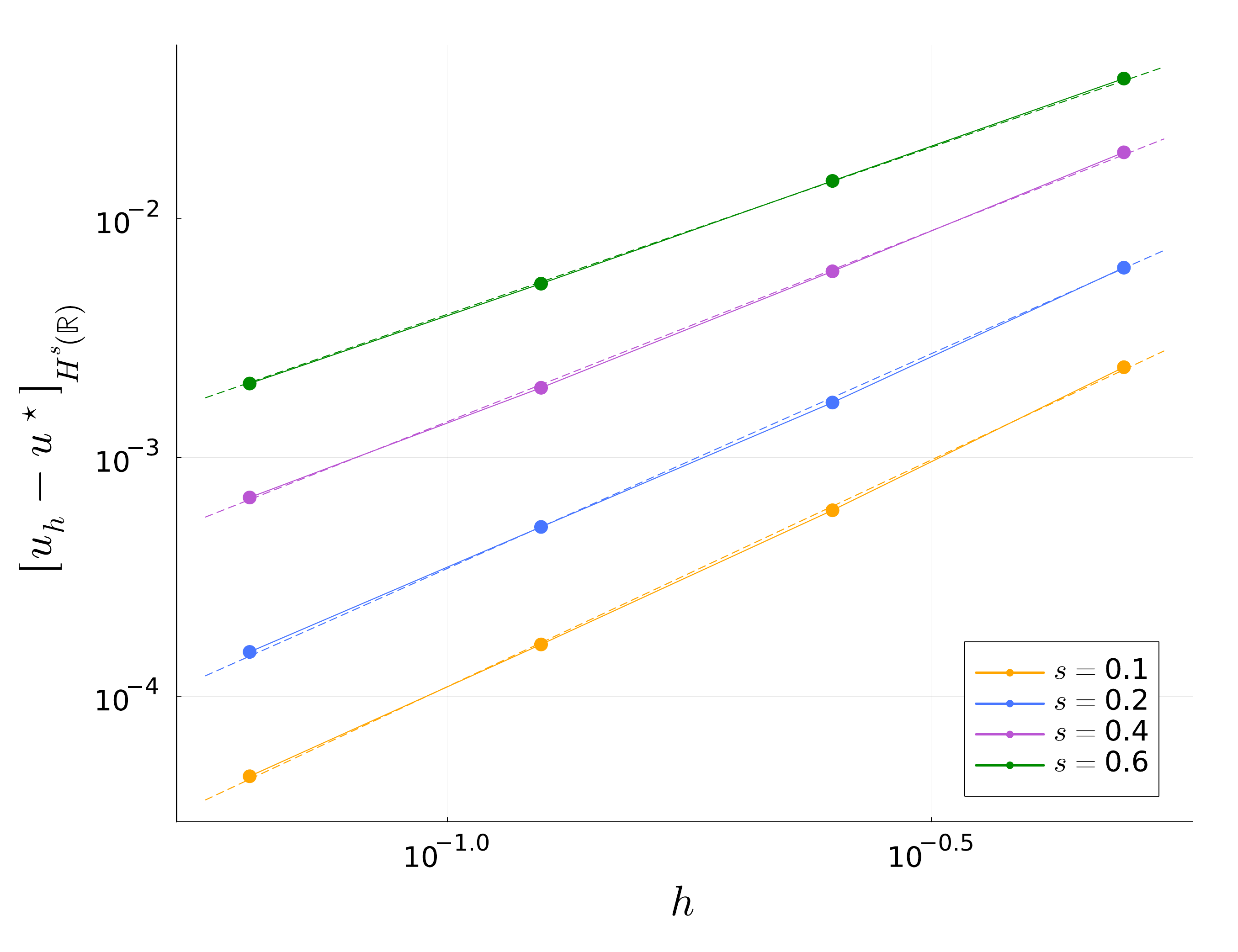}
        \caption{$H^s(\mathbb{R}^d)$-seminorm error between the WFEM approximation and the solution. Dashed lines: $g(h)=Ch^{2-s}$.
        }
        \label{fig: convergence Hs f=1}
    \end{subfigure}
    \hfill
    \begin{subfigure}{0.49 \textwidth}
        \centering
        \includegraphics[width=0.99\textwidth]{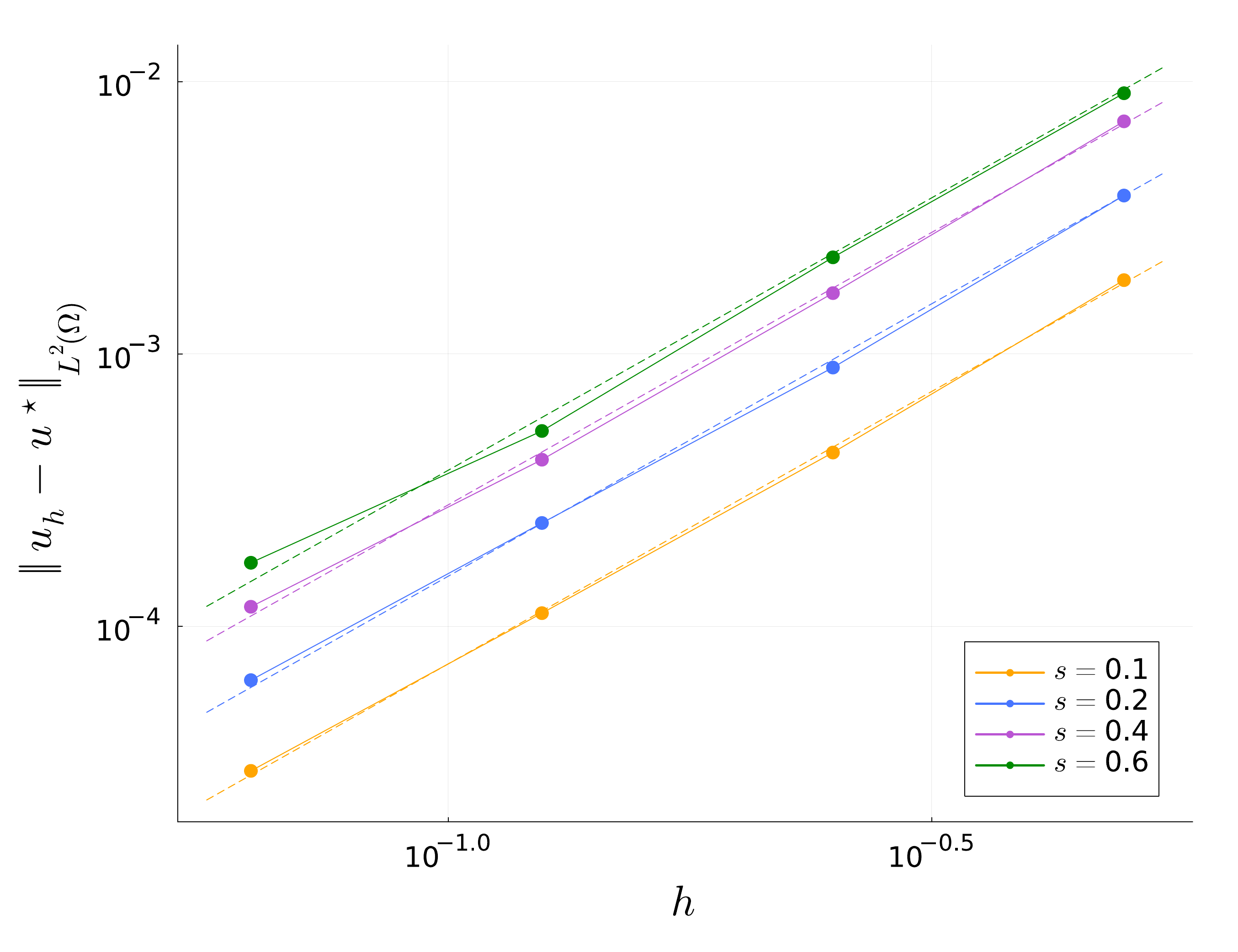}

        \caption{$L^2(\Omega)$ error between the WFEM approximation and the solution. Dashed lines: $g(h)=Ch^2$.
        }
        \label{fig: convergence L2 f=1}
    \end{subfigure}
    \caption{WFEM computational errors. $\Omega = \Omega_h = (-1,1)$, $f=1$,
        and $\delta(x)=1-|x|^4$.
    }
    \label{fig: convergence Hs L2}
\end{figure}

In \Cref{tab: convergence Hs f=1} we report the observed convergence rates for the example corresponding to the slope of the segments in \Cref{fig: convergence Hs L2}\subref{fig: convergence Hs f=1}. The values have been computed using the incremental ratio formula
$(\log [e_h]_{H^s(\Rd)} - [e_{2h}]_{H^s(\Rd)})/(\log h - \log 2h)$, where we recall that $e_h  = u_h - u^\ast$.

\begin{table}[!ht]
    \centering
    \caption{Convergence rates corresponding to \Cref{fig: convergence Hs L2}\subref{fig: convergence Hs f=1}}
    \label{tab: convergence Hs f=1}
    \scriptsize
    \begin{tabular}{lcccc}
        \toprule
        $s$                 & 0.1     & 0.2     & 0.4     & 0.6      \\
        \midrule
        observed $h=2^{-2}$ & 1.9914  & 1.87737 & 1.65643 & 1.42484  \\
        observed $h=2^{-3}$ & 1.86733 & 1.73141 & 1.62081 & 1.430124 \\
        observed $h=2^{-4}$ & 1.83477 & 1.73966 & 1.52709 & 1.38995  \\
        \midrule
        predicted           & 1.9     & 1.8     & 1.6     & 1.4      \\
        \bottomrule
    \end{tabular}
\end{table}

\subsection{\texorpdfstring{$H^s$}{Hs} and \texorpdfstring{$L^2$}{L2} rates when \texorpdfstring{$u(x) = (1 - |x|^2)_+$}{u(x) = (1-|x|2)+}}\label{subsec:bonito}

In \cite{Bonito2019}, the authors consider the solution of \eqref{eq:fracDirProb} given by $u(x) = (1 - |x|^2)_+$, which clearly does not behave like $\dOmega^s$. Using Fourier transform techniques, they compute the corresponding right-hand side
\begin{equation}
    \label{eq:RHS parabola}
    f(x) \coloneqq (-\Delta)^s u(x) = \frac{4^{s} \Gamma(d/2 + s)}{\Gamma(d/2)\,\Gamma(2-s)}\, {}_{2}F_1\!\bigg(\frac{d}{2}+s,\; s-1;\; \frac{d}{2};\; |x|^2\bigg),
\end{equation}
where ${}_{2}F_1$ denotes the Gaussian hypergeometric function. Since $u \not\ge c\,\delta^s$, it follows from \eqref{eq:fpositive} that $f$ must change sign. In fact, for $s \ge 1/2$ and $d=1$, one has $f \notin L^\infty(\Omega)$.

This type of right-hand side has been used for validation of FEM schemes with piece-wise linear functions in \cite{Bonito2019}, where the authors report a rate of $h^{\frac{3}{2}-s}$.
Since $u \in C^\infty(\overline \Omega)$ it is natural that piece-wise linear basis functions yield good approximation.

Since $u/\delta^s \sim \delta^{1-s}$,
\Cref{thm:convergence rate Hs} yields at most a rate of $h^{1-2s}$ when $s \in (0,1/2)$ (the same range where $f \in L^\infty(\Omega)$). This limitation is related to the fact that optimal rates for smooth $f$ rely on $C^{s+\varepsilon}$-regularity to control $H^s$-errors.

Although this goes against the philosophy of our basis choice, we have performed numerical experiments with this right-hand side for completeness. Surprisingly, the observed rates in \Cref{fig: convergence Hs L2 u=parabola} match those in \cite{Bonito2019} for $s \in [1/2,1)$ and even improve them for $s \in (0,1/2)$. An interesting open question is to rigorously verify and justify this behaviour.
In order to approximate the right-hand side of \eqref{eq:main_prob_FE} we perform the experiments in $\Omega = \Omega_h = (-1+\varepsilon,1-\varepsilon)$, $\varepsilon=10^{-10}$.

\begin{figure}[!ht]
    \centering

    \begin{subfigure}{0.49 \textwidth}
        \centering
        \includegraphics[width=0.99\textwidth]{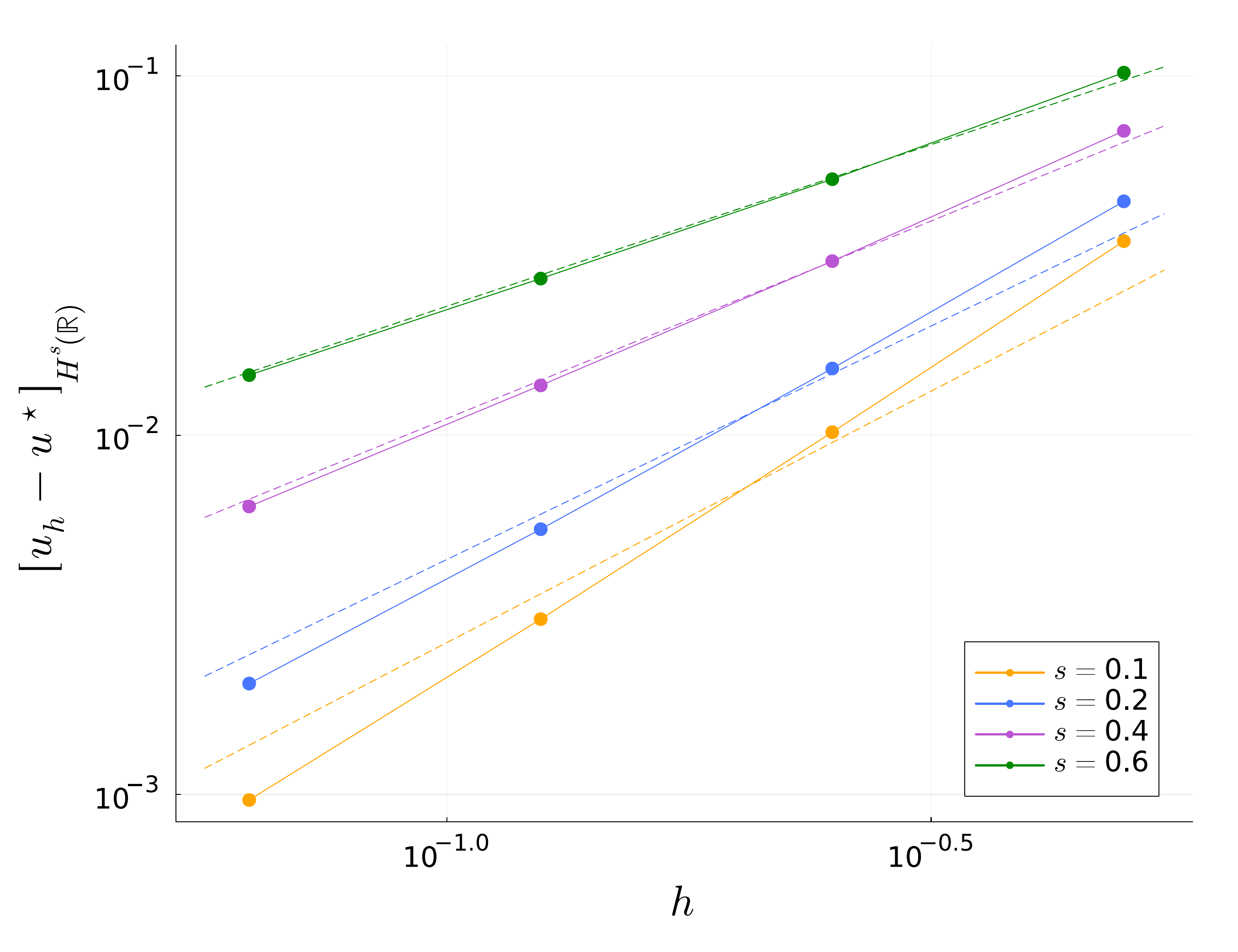}
        \caption{$H^s(\mathbb{R}^d)$-seminorm error between the WFEM approximation and the solution. Dashed lines: $g(h)=Ch^{\frac 3 2 - s}$.
        }
        \label{fig:convergence Hs u=parabola}
    \end{subfigure}
    \hfill
    \begin{subfigure}{0.49 \textwidth}
        \centering
        \includegraphics[width=0.99\textwidth]{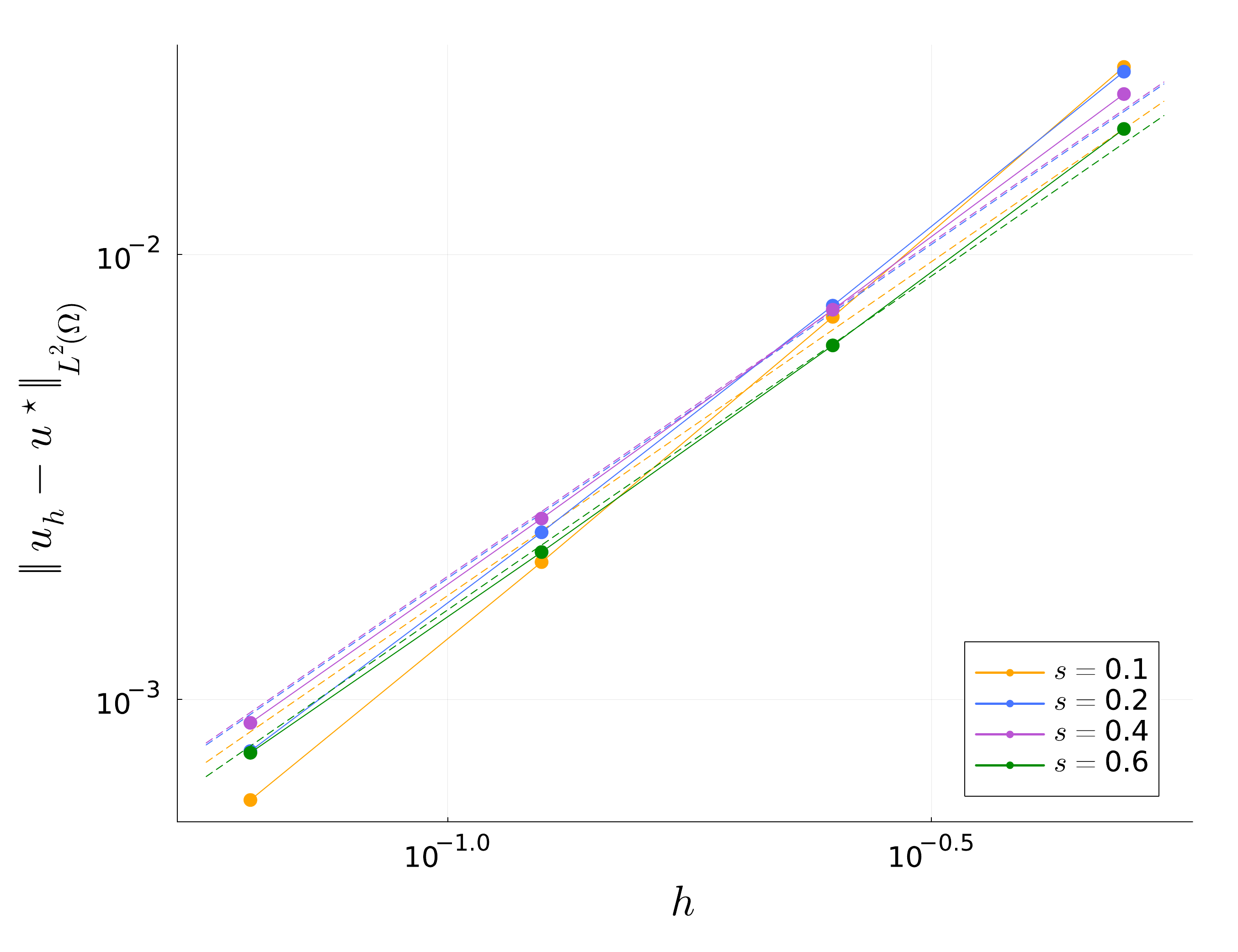}

        \caption{$L^2(\Omega)$ error between the WFEM approximation and the solution. Dashed lines: $g(h)=Ch^{\frac 3 2}$. }
        \label{fig:convergence L2 u=parabola}
    \end{subfigure}
    \caption{
        WFEM computational errors. $\Omega = (-1+\varepsilon,1-\varepsilon)$, $\varepsilon=10^{-10}$,  $f$ given by \eqref{eq:RHS parabola}, $\delta(x)=1-|x|^4$, and $u^* = (1-|x|^2)_+$.
    }
    \label{fig: convergence Hs L2 u=parabola}
\end{figure}

In \Cref{tab: convergence Hs u=parabola}  we report the observed convergent rates for the example corresponding to the slope of the segments in \Cref{fig: convergence Hs L2 u=parabola}\subref{fig:convergence Hs u=parabola}. The values have been computed using the incremental ratio formula $(\log u_{h} - \log u_{2h})/(\log h - \log 2h)$.
\normalcolor

\begin{table}[!ht]
    \centering
    \caption{Convergence rates corresponding to \Cref{fig: convergence Hs L2 u=parabola} \subref{fig:convergence Hs u=parabola}}
    \label{tab: convergence Hs u=parabola}

    \scriptsize
    \begin{tabular}{lcccc}
        \toprule
        $s$                                & 0.1     & 0.2     & 0.4     & 0.6      \\
        \midrule
        observed $h=(1-\varepsilon)2^{-2}$ & 1.7669  & 1.54577 & 1.20472 & 0.985209 \\
        observed $h=(1-\varepsilon)2^{-3}$ & 1.72902 & 1.48647 & 1.14931 & 0.919176 \\
        observed $h=(1-\varepsilon)2^{-4}$ & 1.67254 & 1.42748 & 1.11971 & 0.892579 \\
        \midrule
        predicted in \cite{Bonito2019}     & 1.4     & 1.3     & 1.1     & 0.9      \\
        \bottomrule
    \end{tabular}
\end{table}

\section*{Acknowledgements}
FdT was supported by the Spanish Government through RYC2020-029589-I, PID2021-
127105NB-I00 and CEX2019-000904-S funded by the MICIN/AEI.
The research of DGC was supported by grants RYC2022-037317-I and PID2023-151120NA-I00 from the Spanish Government MCIN/AEI/10.13039/\-501100011033/FEDER, UE.

\printbibliography

\appendix

\section{Functional inequalities}

\begin{lemma}\label{lem:HsWalpha_inclusion}
    Let $s\in (0,1)$ and $\alpha \in (s,1]$. Let $v\in W^{\alpha,\infty}(\R^d)$ such that $\operatorname{supp} v \subset B_R$ for some $R$. Then $v\in H^s(\R^d)$ and
    \[
        [v]_{H^s(\R^d)}\leq C_{d,s,R}\left(\|v\|_{L^\infty(\R^d)} +\frac{1}{\sqrt{\alpha-s}} [v]_{W^{\alpha,\infty}(\R^d)}\right).
    \]
\end{lemma}

\begin{proof}
    We will use the same name for different constants, just keep tracking of their dependence.
    To short the notation, let us write $g(x,y)=\frac{|v(x)-v(y)|^{2}}{|x-y|^{d+2s}}$.  Then,
    \begin{align*}
        [v]_{H^s(\R^d)}^2 & =\iint \limits_{B_R\times B_R} g(x,y) \dx\dy + 2 \iint\limits_{\substack{B_R\times (\R^d \setminus B_R) \\ |x-y|\leq1}} g(x,y) \dx\dy+ 2 \iint\limits_{\substack{B_R\times (\R^d \setminus B_R)\\ |x-y|>1}} g(x,y) \dx\dy \\
                          & \defeq I_1+I_2+I_3.
    \end{align*}
    Let us estimate each term above. First,
    \begin{align*}
        I_1=[v]_{W^{\alpha,\infty}(\R^d)}^2\iint \limits_{B_R\times B_R} \frac{|x-y|^{2\alpha}}{|x-y|^{d+2s}} \dx\dy\leq C_{d,R}[v]_{W^{\alpha,\infty}(\R^d)}^2 \int_0^{2R} r^{2\alpha-1-2s dr}\leq \frac{C_{d,R}}{\alpha-s}[v]_{W^{\alpha,\infty}(\R^d)}^2.
    \end{align*}
    Similarly, we have that
    \begin{align*}
        I_2\leq 2 \iint\limits_{B_{R+1}\times B_{R+1}} g(x,y) \dx\dy \leq \frac{C_{d,R}}{\alpha-s}[v]_{W^{\alpha,\infty}(\R^d)}^2.
    \end{align*}
    Finally, we deduce
    \begin{align*}
        I_3\leq 4\|v\|_{L^\infty(\R^d)}^2\int_{B_R}\int_{|x-y|>1} \frac{1}{|x-y|^{d+2s}}\dx \dy = C_{d,s,R}\|v\|_{L^\infty(\R^d)}^2.
    \end{align*}
    The proof is completed by observing that $(a^2+b^2)^{\frac{1}{2}}\leq |a|^{\frac{1}{2}}+|b|^{\frac{1}{2}}$ for all $a,b\in \R$.
\end{proof}

\end{document}